\documentclass{amsart}
\usepackage[utf8]{inputenc}

\usepackage{amsmath,amsthm,amssymb,amsfonts,enumerate,color}
\usepackage{mathrsfs}
\usepackage{amstext,amsxtra}
\usepackage[margin=1.5in]{geometry}

\usepackage{tikz}
\usetikzlibrary{matrix,arrows,calc,intersections,fit}
\usetikzlibrary{decorations.markings}
\usepackage{tikz-cd}
\usepackage{textcomp}
\usepackage[colorlinks,urlcolor=black,linkcolor=blue,citecolor=blue,hypertexnames=false]{hyperref}
\allowdisplaybreaks
\usepackage{pgf,tikz}
\usepackage{stmaryrd}
\usepackage{bm}

\newtheorem{thm}{Theorem}[section]
\newtheorem{prop}[thm]{Proposition}

\newtheorem{defi}[thm]{Definition}
\newtheorem{lem}[thm]{Lemma}
{\theoremstyle{remark}
\newtheorem{exa}[thm]{Example}
\newtheorem{rem}[thm]{Remark}}

\newcommand{\mult}{\operatorname{mult}}
\newcommand{\aut}{\operatorname{Aut}}
\newcommand{\conj}{\operatorname{conj}}
\newcommand{\id}{\operatorname{id}}

\newcommand{\stab}{\operatorname{Stab}}
\newcommand{\sym}{\operatorname{Sym}}
\newcommand{\symc}{\operatorname{SymC}}

\newcommand{\sal}{\mathcal{S}}

\newcommand{\rl}{\mathcal{R}}

\newcommand{\cl}{\mathcal{C}}

\newcommand{\fl}{\mathcal{F}}

\newcommand{\rb}{\mathbb{R}}

\newcommand{~}{\quad}
\newcommand{\cb}{\mathbb{C}}
\newcommand{\nb}{\mathbb{N}}

\newcommand{\pb}{\mathbb{P}}

\newcommand{\ak}{\mathfrak{a}}

\newcommand{\undl}{\underline}

\begin{document}

\title{On the lower bounds for real double Hurwitz numbers}

\author{Yanqiao Ding}

\thanks{Corresponding author: yqding@zzu.edu.cn}

\address{School of Mathematics and Statistics, Zhengzhou University, Zhengzhou, 450001, China}

\email{yqding@zzu.edu.cn}



\date{\today}

\begin{abstract}
As the real counterpart of double Hurwitz number,
the real double Hurwitz number depends on
the distribution of real branch points.
We consider the problem of asymptotic growth of real and
complex double Hurwitz numbers. We provide a lower bound for real double Hurwitz numbers based on the tropical computation of real double Hurwitz numbers.
By using this lower bound and J. Rau's result ( Math. Ann. 375(1-2): 895-915, 2019),
we prove the logarithmic equivalence of real and complex Hurwitz numbers.
\end{abstract}

\maketitle

\section{Introduction}
The structure and approach of this paper follows \cite{rau2019} closely.
Counting the ramified covers of $\cb P^1$ by a
genus $g$ surface with specified
ramification profiles over a fixed set of points
is a classical problem of enumerative geometry.
The answer to such enumerative problem is called
the Hurwitz number. The Hurwitz number is equivalent to enumerating
the factorizations of the identity into a product
of elements of the symmetric group $\sal_d$
with given cycle types \cite{cm-2016,hurwitz-1891}.
Hurwitz numbers are interesting geometric invariants connecting the
geometry of algebraic curves, combinatorics, tropical geometry, the
representation of symmetric groups and random matrix
models \cite{bbm-2011,cjm-2010,cm-2016,dyz-2017,gjv-2005}. In the recent two decades,
many deep relationships between Hurwitz
numbers and mathematical physics were also found
\cite{lzz-2000,op-2006}.
There is a particular type of Hurwitz numbers which
arouses many mathematicians' interests.
The double Hurwitz number $H^\cb_g(\lambda,\mu)$ counts
covers of $\cb P^1$ by genus $g$ surface with
ramification profiles $\lambda$, $\mu$ over $0$, $\infty$
and simple ramification over other branch points,
where $\lambda$ and $\mu$ are two partitions of an
integer $d\geq1$.
For a partition $\lambda$,
we denote by $l(\lambda)$ the number of parts
of $\lambda$ and call it the length of $\lambda$.
The sum of the parts of $\lambda$ is denoted by $|\lambda|$.
There are many results about the structure of
double Hurwitz numbers such as the polynomiality of
the generating function and the wall-crossing
formulas \cite{cjm-2011,gjv-2005,johnson-2015,ssv-2008}.

In this paper, we consider the real version of double
Hurwitz number.
A \textit{real structure} $\tau$ for
a cover $\pi:C\to\cb P^1$ is an anti-holomorphic involution
such that $\pi\circ\tau=\conj\circ\pi$,
where $\conj$ is the standard complex conjugation.
A pair $(\pi,\tau)$ consisting of a ramified cover $\pi:C\to\cb P^1$
and a real structure $\tau$ is called a \textit{real ramified cover}.
The real double Hurwitz number counts real ramified covers of $\cb P^1$ by genus $g$ surfaces with
particular ramification profiles over $0$, $\infty$
and simple ramification over other branch points.
Note that the set of simple branch points
of a real ramified cover consists of real points in $\rb P^1\setminus\{0,\infty\}$ and complex conjugated pairs.
In the following, we only consider the case that
all the simple branch points are real points in $\rb P^1\setminus\{0,\infty\}$,
and suppose that $s$ of these simple branch points are in the positive half axis of $\rb P^1\setminus\{0,\infty\}$.
Let $H^\rb_g(\lambda,\mu;s)$ denote the real double Hurwitz number counting
real ramified covers of $\cb P^1$ by genus $g$ surface with
ramification profiles $\lambda$, $\mu$ over $0$, $\infty$
and simple ramification over other branch points.
The real double Hurwitz number $H^\rb_g(\lambda,\mu;s)$ depends on the number $s$ of positive real simple branch points.
It is a common phenomenon in real enumerative geometry that the number of real solutions for
a enumerative problem depends on the positions of the point constraints \cite{iks2003,wel2005a,wel2005b}.
In real enumerative
geometry, it is important to find the lower bounds
for the real enumerative problems and to analyse the
properties of these lower bounds.
In the study
of real algebraic curves in real surfaces passing
through certain real points,
signed counts of real solutions
which are called the Welschinger invariants
provide such lower bounds
\cite{gz2018,iks2013b,ks-2015,wel2005a,wel2005b}.
This method is also valid in the study of counting
real covers. Itenberg and Zvonkine
\cite{iz-2018} found such a signed count
of real polynomials
and proved that the signed count of real polynomials
is logarithmically equivalent to the count of
complex polynomials under certain parity conditions.
El Hilany and Rau \cite{er-2019} found
that the construction of Itenberg and Zvonkine
also works for counting real simple rational
functions $\frac{f(x)}{x-p}$, $f(x)\in\rb[x]$, $p\in\rb$.
How to generalize Itenberg and Zvonkine's
signed count to a more general situation is
still unknown. Rau \cite{rau2019} found a lower bound
for real double Hurwitz numbers, and proved the
logarithmic equivalence of real double Hurwitz numbers
and complex double Hurwitz numbers under certain parity conditions.

In this paper, we continue the study on the asymptotic growth of real double Hurwitz numbers when the degree is increased and only
simple ramification points are added.
We prove the logarithmic equivalence of real and classical Hurwitz numbers. Let
\begin{align*}
    h^\cb_{g,\lambda,\mu}(m)&=
    H^\cb_g((\lambda,1^{m}),(\mu,1^{m})),\\
    h^\rb_{g,\lambda,\mu}(m)&=\inf\{
    H^\rb_g((\lambda,1^{m}),(\mu,1^{m});0),
    \ldots,H^\rb_g((\lambda,1^{m}),(\mu,1^{m});r(m))\},
\end{align*}
where $(\lambda,1^m)$ stands for adding $m$ ones to $\lambda$,
and $r(m)=l(\lambda)+l(\mu)+2m+2g-2$.
\begin{thm}\label{thm:main}
Fix $g\in\nb$, and partitions $\lambda$, $\mu$ with
$|\lambda|=|\mu|$. Then
$h^\rb_{g,\lambda,\mu}(m)$ and
$h^\cb_{g,\lambda,\mu}(m)$ are logarithmically
equivalent:
$$
\log h^\rb_{g,\lambda,\mu}(m)\sim
2m\log m
\sim\log
h^\cb_{g,\lambda,\mu}(m), \text{ as } m\to\infty.
$$
\end{thm}

For the enumerative problem concerning real rational algebraic curves in real algebraic surfaces,
Itenberg, Kharlamov and Shustin \cite{iks2004,iks2007} showed that Welschinger invariants are logarithmically equivalent to the Gromov-Witten invariants.
Shustin \cite{shustin2015} proved the logarithmic equivalence of higher genus Welschinger invariants
and Gromov-Witten invariants.
In the following, we recall the main asymptotic statements
about real Hurwitz numbers from
\cite{er-2019,iz-2018,rau2019}.
Denote by $S^{even}_{pol}(\lambda_1,\ldots,\lambda_k)$
(resp. $S^{odd}_{pol}(\lambda_1,\ldots,\lambda_k)$)
the signed counts of real polynomials of even (resp. odd)
degree with reduced ramification profiles $\lambda_1,\ldots,\lambda_k$ defined in \cite{iz-2018}.
\begin{thm}[{\cite[Theorem $5$]{iz-2018}}]
Assume that in each partition $\lambda_i$, $i\in\{1,\ldots,k\}$, every even number appears an even
number of times and at most one odd number appears an odd number of times. Then we have
$$
\log |s^{even}_{pol}(m)|\sim m\log m\sim\log
h^{even,\cb}_{pol}(m), \text{ as } m\to\infty,
$$
where $s^{even}_{pol}(m)=S^{even}_{pol}(\lambda_1,\ldots,\lambda_k,1^{m})$,
and $h^{even,\cb}_{pol}(m)$ is the corresponding counts of complex polynomials of even degree.

Assume that in each partition $\lambda_i$, $i\in\{1,\ldots,k\}$, at most one even number appears
an odd number of times and at most one odd number appears
an odd number of times. Then we have
$$
\log |s^{odd}_{pol}(m)|\sim m\log m\sim\log
h^{odd,\cb}_{pol}(m), \text{ as } m\to\infty,
$$
where $s^{odd}_{pol}(m)=S^{odd}_{pol}(\lambda_1,\ldots,\lambda_k,1^{m})$,
and $h^{odd,\cb}_{pol}(m)$ is the corresponding counts of complex polynomials of odd degree.
\end{thm}
Let $S_{rat}(\lambda_1,\ldots,\lambda_k)$
denote the signed counts of real simple
rational functions with reduced ramification profiles
$\lambda_1,\ldots,\lambda_k$ defined in \cite{er-2019}.
\begin{thm}[{\cite[Theorem $1.3$]{er-2019}}]
Assume that in each partition $\lambda_i$, $i\in\{1,\ldots,k\}$, at most one even number appears
an odd number of times and at most one odd number appears
an odd number of times. Then we have
$$
\log |s_{rat}(m)|\sim m\log m\sim\log
h^\cb_{rat}(m), \text{ as } m\to\infty, \text{ and }\sum|\lambda_i|+m\equiv0\mod 2,
$$
where $s_{rat}(m)=S_{rat}(\lambda_1,\ldots,\lambda_k,1^{m})$, and $h^\cb_{rat}(m)$ is the corresponding counts of complex simple rational functions.
When $\sum|\lambda_i|+m\equiv1\mod 2$,
if the partitions
$\lambda_i$, $i\in\{1,\ldots,k\}$, satisfy the above condition and an extra parity condition,
$|s_{rat}(m)|$ is also logarithmically equivalent to $h^\cb_{rat}(m)$.
\end{thm}

A \textit{tropical cover} is a continuous map from a
connected metric graph to $\rb\cup\{\pm\infty\}$
satisfying certain conditions (see Definition $\ref{def:tro-cover}$ for more details).
There is a multiplicity associated with each tropical cover in \cite{gpmr-2015,mr-2015}.
Markwig and Rau \cite{mr-2015} gave a tropical
interpretation of real double Hurwitz numbers via the weighted count of tropical covers.
That is the tropical cover with odd multiplicity which is called \textit{zigzag cover}.
Rau \cite{rau2019} observed that the number of zigzag covers, $Z_g(\lambda,\mu)$, is independent of the number of real positive simple branch points.
From Markwig and Rau's tropical
computation of real double Hurwitz numbers
in \cite{mr-2015}, $Z_g(\lambda,\mu)$ is a lower bound
for real double Hurwitz numbers.
\begin{thm}[{\cite[Theorem $5.10$]{rau2019}}]
Fix $g\in\nb$, and partitions $\lambda$, $\mu$ with
$|\lambda|=|\mu|$. Assume that
the number of odd elements which appear an
odd number of times in $\lambda$ plus the number of odd elements which appear an odd number
of times in $\mu$ is $0$ or $2$.
Then $z_{g,\lambda,\mu}(2m)$ and
$h^\cb_{g,\lambda,\mu}(2m)$ are logarithmically
equivalent,
$$
\log z_{g,\lambda,\mu}(2m)\sim 4m\log m\sim\log
h^\cb_{g,\lambda,\mu}(2m), \text{ as } m\to\infty,
$$
where $z_{g,\lambda,\mu}(2m)=Z_g((\lambda,1^{2m}),(\mu,1^{2m}))$.
\end{thm}

\begin{rem}
The assumption on odd elements in
$\lambda$ and $\mu$ in
\cite[Theorem $5.10$]{rau2019} is a necessary
condition to guarantee the existence of zigzag
covers.
\end{rem}

We provide a new lower bound for real double Hurwitz numbers to prove our main result.
A \textit{real tropical cover} is a pair consisting of
a tropical cover and a colouring of it.
In the correspondence theorem \cite{mr-2015}, the real double Hurwitz number $H^\rb_g(\lambda,\mu;s)$
is expressed as a weighted sum over isomorphism classes of real tropical covers.
From \cite[Proposition $4.8$ and Lemma $4.14$]{rau2019},
we know that the set of zigzag covers is exactly the set of
tropical covers admitting a colouring compatible with
any splitting of real branch points.
We have an observation that
$H^\rb_g(\lambda,\mu;s)=H^\rb_g(\lambda,\mu;r-s)$,
where $r$ is the number of real simple branch
points of the ramified covers counted by
$H^\rb_g(\lambda,\mu;s)$ (see Proposition $\ref{prop:rHurwitz-symmetry}$).
The integer $r=l(\lambda)+l(\mu)+2g-2$ is determined
by Riemann-Hurwitz formula. In order to
find a lower bound for $H^\rb_g(\lambda,\mu;s)$,
$0\leq s\leq r$, we only need to find a lower bound
for $H^\rb_g(\lambda,\mu;s)$,
$\left \lceil\frac{r}{2}\right\rceil\leq s\leq r$.
Our idea is to characterize a set of tropical covers such that
for any $s\geq\left \lceil\frac{r}{2}\right\rceil$
the tropical cover in this set admits a colouring with $s$ positive branch points.
From \cite{mr-2015}, we know that the number of these covers is a lower bound for real double
Hurwitz numbers. It is easy to see that this lower
bound is bigger than the number of zigzag covers,
because some tropical covers with even multiplicity
are also counted.
We call this new lower bound
the {\it effective number},
and denote it by $E_g(\lambda,\mu)$
(see Definition $\ref{def:effective-number}$).
We have
$$
Z_g(\lambda,\mu)\leq E_g(\lambda,\mu)\leq H^\rb_g(\lambda,\mu;s)\leq
H^\cb_g(\lambda,\mu).
$$
\begin{thm}\label{thm:asymptotic}
Fix $g\in\nb$, and partitions $\lambda$, $\mu$ with
$|\lambda|=|\mu|$.
Suppose that the sum of odd numbers which appear odd number
of times in $\lambda$ is greater than or equal to
the sum of odd numbers which appear odd number
of times in $\mu$,
then $e_{g,\lambda,\mu}(m)$ and
$h^\cb_{g,\lambda,\mu}(m)$ are logarithmically
equivalent for even $m$:
$$
\log e_{g,\lambda,\mu}(m)\sim
2m\log m\sim\log
h^\cb_{g,\lambda,\mu}(m), \text{ as }m\to\infty \text{ for even } m,
$$
where $e_{g,\lambda,\mu}(m)=E_g((\lambda,1^{m}),(\mu,1^{m}))$.
\end{thm}

\begin{rem}
Theorem $\ref{thm:main}$ is a straightforward application
of Theorem $\ref{thm:asymptotic}$ and the symmetry of
$H^\rb_g(\lambda,\mu;s)$ in $\lambda$ and $\mu$
(see Proposition $\ref{prop:rHurwitz-symmetry1}$).
\end{rem}

\section{Double Hurwitz numbers}
In this section, we recall some facts
about double Hurwitz numbers.
The readers may refer to \cite{cjm-2010,cm-2016,mr-2015}
for more details.
\subsection{Complex double Hurwitz numbers}
\label{subsec:complex-Hurwitz}

We fix two integers $d\geq1$, $g\geq0$,
and let $\lambda$ and $\mu$ be two partitions of
$d$.
Fix a collection of
$r=l(\lambda)+l(\mu)+2g-2$ points $\undl p=\{p_1,\ldots,p_r\}\subset\cb P^1\setminus
\{0,\infty\}$.
\begin{defi}
A complex Hurwitz cover of type $(g,\lambda,\mu,\undl p)$
is a degree $d$ holomorphic map $\pi:C\to\cb P^1$
such that:
\begin{itemize}
    \item $C$ is a connected Riemann surface of genus $g$;
    \item $\pi$ ramifies with profiles $\lambda$
    and $\mu$ over $0$ and $\infty$ respectively;
    \item all the points in $\undl p$ are simple branch points of $\pi$;
    \item $\pi$ is unramified everywhere else.
\end{itemize}
\end{defi}
An isomorphism of two complex Hurwitz covers
$\pi_1:C_1\to\cb P^1$ and $\pi_2:C_2\to\cb P^1$
is an isomorphism of Riemann surfaces $\varphi:C_1\to C_2$
such that $\pi_1=\pi_2\circ\varphi$.
The complex double Hurwitz number is
$$
H^\cb_g(\lambda,\mu)=\sum_{[\pi]}\frac{1}{|\aut^\cb(\pi)|},
$$
where we sum over all isomorphism classes of
complex Hurwitz covers of type $(g,\lambda,\mu,\undl p)$.
It is a classical result that this number does not depend
on the positions of $\undl p$ \cite{cm-2016,hurwitz-1891}.

There is also an equivalent way to define complex double
Hurwitz number via symmetric groups.
Let $\sal_d$ denote the symmetric group of order $d$.
We denote by $\cl(\sigma)\vdash\sal_d$ the cycle type of
$\sigma\in\sal_d$.
Let $d$, $g$, $\lambda$ and $\mu$ be as above.
\begin{defi}
A factorization of type $(g,\lambda,\mu)$ is a tuple
$(\sigma_1,\tau_1,\ldots,\tau_r,\sigma_2)$
of elements of $\sal_d$ such that:
\begin{itemize}
    \item $\sigma_2\cdot\tau_r\cdot\cdots\cdot\tau_1\cdot\sigma_1=\id$;
    \item $r=l(\lambda)+l(\mu)+2g-2$;
    \item $\cl(\sigma_1)=\lambda$, $\cl(\sigma_2)=\mu$,
    $\cl(\tau_i)=(2,1,\ldots,1)$, $i=1,\ldots,r$;
    \item the subgroup generated by $\sigma_1$,
    $\sigma_2$, $\tau_1,\ldots,\tau_r$ acts transitively
    on the set $\{1,\ldots,d\}$.
\end{itemize}
\end{defi}
We denote by $\fl(g,\lambda,\mu)$ the set of
all factorizations of type $(g,\lambda,\mu)$.
\begin{thm}[Hurwitz \cite{cm-2016,hurwitz-1891}]
Let $d\geq1$, $g\geq0$ be two integers,
$\lambda$ and $\mu$ be two partitions of $d$.
Then
$$
H^\cb_g(\lambda,\mu)=\frac{1}{d!}|\fl(g,\lambda,\mu)|.
$$
\end{thm}

\subsection{Real double Hurwitz numbers}
Let $g$, $d$, $\lambda$ and $\mu$ be as above.
In the rest of this paper, we assume that
the set of simple branch points $\undl p=\{p_1,\ldots,p_r\}$
is a subset of $\rb P^1\setminus\{0,\infty\}$, and satisfies
$p_1<\ldots<p_r$.

\begin{defi}
A real Hurwitz cover of type $(g,\lambda,\mu,\undl p)$
is a tuple $(\pi,\tau)$ such that
\begin{itemize}
    \item $\pi:C\to\cb P^1$ is a complex Hurwitz cover of type $(g,\lambda,\mu,\undl p)$;
    \item $\tau:C\to C$ is an anti-holomorphic involution such that $\pi\circ\tau=\conj\circ\pi$.
\end{itemize}
\end{defi}
An isomorphism of two real Hurwitz covers
$(\pi_1:C_1\to\cb P^1,\tau_1)$
and $(\pi_2:C_2\to\cb P^1,\tau_2)$
is an isomorphism of complex Hurwitz covers
$\varphi:C_1\to C_2$
such that $\varphi\circ\tau_1=\tau_2\circ\varphi$.
Let $s=|\undl p\cap\rb^+|$.
The real double Hurwitz number is
$$
H^\rb_g(\lambda,\mu;s)=\sum_{[(\pi,\tau)]}
\frac{1}{|\aut^\rb(\pi,\tau)|},
$$
where we sum over all isomorphism classes of
real Hurwitz covers of type $(g,\lambda,\mu,\undl p)$.
Note that the integer $H^\rb_g(\lambda,\mu;s)$
depends on the
positions of points in $\undl p$.

The symmetric group can also be used to study real double
Hurwitz number \cite{cadoret-2005,gpmr-2015}.
In the Appendix, we give an equivalent
description of real double Hurwitz number
using symmetric group (see Lemma $\ref{lem:realDH1}$).

\begin{prop}\label{prop:rHurwitz-symmetry}
Let $d\geq1$, $g\geq0$ be two integers,
and $\lambda$, $\mu$ be two partitions
of $d$. Suppose that $0\leq s\leq r$, where
$r=l(\lambda)+l(\mu)+2g-2$.
Then
$$
H^\rb_g(\lambda,\mu;s)=H^\rb_g(\lambda,\mu;r-s).
$$
\end{prop}

\begin{proof}
For any real Hurwitz
cover $(\pi,\tau)$ of type $(g,\lambda,\mu,\undl p)$,
we compose the cover map $\pi$ with the map
$$
\begin{aligned}
    -1:\cb P^1&\to\cb P^1\\
    z&\mapsto -z,
\end{aligned}
$$
then we obtain a real Hurwitz cover $(\pi',\tau)$ of
type $(g,\lambda,\mu,\undl{-p})$,
where $\undl{-p}=\{-p_1,-p_2,\ldots,-p_r\}$.
Suppose that $s=|\undl p\cap\rb^+|$.
It is easy to see that
$|\aut^\rb(\pi,\tau)|=|\aut^\rb(\pi',\tau)|$ and
$r-s=|\undl{-p}\cap\rb^+|$.
Since the map $-1:\cb P^1\to\cb P^1$ induces a bijection between the set of
all isomorphism classes of real Hurwitz covers of
type $(g,\lambda,\mu,\undl p)$ and the set of type
$(g,\lambda,\mu,\undl{-p})$, we obtain
that $H^\rb_g(\lambda,\mu;s)=H^\rb_g(\lambda,\mu;r-s)$.
\end{proof}

\begin{prop}\label{prop:rHurwitz-symmetry1}
Let $d$, $g$, $\lambda$, $\mu$, $s$ and $r$ be the same as
Proposition $\ref{prop:rHurwitz-symmetry}$.
Then
$$
H^\rb_g(\lambda,\mu;s)=H^\rb_g(\mu,\lambda;s).
$$
\end{prop}

\begin{proof}
We compose any real Hurwitz
cover $(\pi,\tau)$ of type $(g,\lambda,\mu,\undl p)$
with the map
$$
\begin{aligned}
    \psi:\cb P^1&\to\cb P^1\\
    z&\mapsto \frac{1}{z},
\end{aligned}
$$
then we obtain a real Hurwitz cover $(\pi',\tau)$ of
type $(g,\mu,\lambda,\undl{p}')$,
where $\undl{p}'=\{\frac{1}{p_1},\ldots,\frac{1}{p_r}\}$.
A same argument as the proof of Proposition $\ref{prop:rHurwitz-symmetry}$ implies that
$H^\rb_g(\lambda,\mu;s)=H^\rb_g(\mu,\lambda;s)$.
\end{proof}

\subsection{Tropical double Hurwitz numbers}
Let us recall some definitions first.
The readers may refer to \cite{cjm-2010,gpmr-2015,mr-2015,rau2019} for more details.
Let $\Gamma$ be a connected graph without $2$-valent vertices.
We call $1$-valent vertices of $\Gamma$ the
{\it leaves},
and the higher-valent vertices are called the
{\it inner vertices}.
The edges adjacent to a $1$-valent vertex are called
{\it ends}. Edges which are not ends are
called {\it inner edges}. By $\Gamma^\circ$ we denote
the subgraph obtained by removing the $1$-valent vertices
of $\Gamma$.
The number $g=b_1(\Gamma)$, the first Betti number,
is called the {\it genus} of $\Gamma$.
A {\it tropical curve} $C$ is a connected metric graph
without $2$-valent vertices
such that the length of an end is $\infty$,
and the length $\ell(e)\in\rb$ of an inner edge $e$ is finite.
Note that the tropical projective line $T\pb^1$
is considered as $\rb\cup\{\pm\infty\}$.
Throughout this paper, except $T\pb^1$, we only consider
graph without two-valent vertices.
An isomorphism $\varPhi:C_1\to C_2$ of two
tropical curves is an isometric homeomorphism
$\varPhi:C_1^\circ\to C_2^\circ$.
\begin{defi}\label{def:tro-cover}
A tropical cover $\varphi:C\to T\pb^1$ is a continuous
map satisfying:
\begin{itemize}
    \item the image of any inner vertex of $C$ under
    $\varphi$ is contained in $\rb$. Let $\undl x$ denote
    the set of images of inner vertices of $C$, and
    we call $\undl x$ the inner vertices of $T\pb^1$.
    \item $\varphi^{-1}(+\infty)\neq\emptyset$,
    $\varphi^{-1}(-\infty)\neq\emptyset$, and
    $\varphi^{-1}(+\infty)\cup\varphi^{-1}(-\infty)$
    is the set of leaves of $C$.
    \item $\varphi$ is a piecewise linear map: for any
    edge $e$ of $C$, we interpret $e$ as an interval
    $[0,\ell(e)]$, then there is a positive
    integer $\omega(e)$ such that
    $\varphi(t)=\pm\omega(e)t+\varphi(0)$, $\forall t\in[0,\ell(e)]$.
    The integer $\omega(e)$ is called the \textit{weight} of $e$.
    \item For any vertex $v\in C$, we choose an
    edge $e'\subset T\pb^1$ adjacent to $\varphi(v)$.
    Then the integer
    $$
    \deg(\varphi,v):=\sum_{
    e\text{ edge of } C\atop
    v\in e,\varphi(e)=e'}
    \omega(e)
    $$
    does not depend on the choice of $e'$.
    This is called the \textit{balancing} or \textit{harmonicity condition}.
\end{itemize}
\end{defi}
Let $\varphi:C\to T\pb^1$ be a tropical cover. The sum
$$
\deg(\varphi):=\sum_{
    e\text{ edge of } C\atop
    \varphi(e)=e'}\omega(e)
$$
is independent of $e'$. The integer $\deg(\varphi)$
is the degree of $\varphi$.
Let $d\geq1$, $g\geq0$, and $\lambda$, $\mu$ be two partitions of $d$.
Suppose $r=l(\lambda)+l(\mu)+2g-2>0$. Fix $r$
points $\undl x=\{x_1,\ldots,x_r\}\subset\rb$
satisfying $x_1<\ldots<x_r$.
\begin{defi}
A tropical cover $\varphi:C\to T\pb^1$ of type
$(g,\lambda,\mu,\undl x)$ is a tropical cover
of degree $d$ such that
\begin{itemize}
    \item $C$ is a tropical curve of genus $g$;
    \item the tuple of weights of ends adjacent to
    leaves mapping to $-\infty$ is $\lambda$,
    the tuple of weights of ends adjacent to
    leaves mapping to $+\infty$ is $\mu$;
    \item each $x_i\in\undl x$ is the image of an inner
    vertex of $C$.
\end{itemize}
\end{defi}
An {\it isomorphism} of two tropical covers
$\varphi_1:C_1\to T\pb^1$ and $\varphi_2:C_2\to T\pb^1$ is
an isomorphism $\psi:C_1\to C_2$ of tropical curves
such that $\varphi_1=\varphi_2\circ\psi$.
The multiplicity of a tropical cover $\varphi$
is defined to be
$$
\mult^\cb(\varphi):=\frac{1}{|\aut(\varphi)|}\prod_{e\text{ inner} \atop\text{edge of }C}\omega(e).
$$
A symmetric cycle is a pair of inner edges of the same
weight and adjacent to the same two vertices.
A symmetric fork is a pair of ends of the same weight
adjacent to a same vertex.
Note that the group of automorphisms of a tropical cover $\varphi:C\to T\pb^1$
is induced by the interchange of two edges in symmetric cycles and
symmetric forks.
\begin{thm}[{\cite[Theorem $5.28$]{cjm-2010}}]
The complex Hurwitz number $H^\cb_g(\lambda,\mu)$ is
equal to
$$
H_g^\cb(\lambda,\mu)=\sum_{[\varphi]}\mult^\cb(\varphi),
$$
where we sum over all isomorphism classes $[\varphi]$
of tropical covers of type $(g,\lambda,\mu,\undl x)$.
\end{thm}

\subsection{Real tropical double Hurwitz numbers}
In this subsection, we will review some basic facts about
real tropical double Hurwitz numbers.
Our notations follow closely with \cite[Section $3$]{rau2019}.
Let $\varphi:C\to T\pb^1$ be a tropical cover.
An edge of even or odd weight is called
even or odd edge,
respectively. A symmetric cycle (resp. fork)
consisting of a pair of even or odd inner edges (resp. ends)
is called an even or odd symmetric cycle (resp. fork).
The following notations will be used in the rest of this paper.
\begin{itemize}
    \item $\sym(\varphi)$ denotes the set of symmetric cycles and
symmetric odd forks.
    \item $\symc(\varphi)\subset\sym(\varphi)$ is the set of symmetric cycles.
    \item For $T\subset\sym(\varphi)$, $C\setminus T^\circ$ is
the subgraph of $C$ obtained by removing the interior
of the edges contained in $T$.
    \item $E(T)$ is the set of even inner edges in
$C\setminus T^\circ$.
\end{itemize}

\begin{defi}[{\cite[Definition $3.3$]{rau2019}}]
A colouring $\rho$ of a tropical cover $\varphi:C\to T\pb^1$
consists of a choice of subset $T_{\rho}\subset\sym(\varphi)$
and a choice of a colour red or blue for every
component of the subgraph of even edges of
$C\setminus T_{\rho}^\circ$.
\end{defi}
A \textit{real tropical cover} is a tuple
$(\varphi:C\to T\pb^1,\rho)$ of a tropical cover $\varphi$
and a colouring $\rho(\varphi)$ of $\varphi$.
An isomorphism of two real tropical covers is an
isomorphism of these two tropical covers that respects
the colouring.
The real multiplicity of a real tropical cover is
$$
\mult^\rb(\varphi,\rho):=2^{|E(T_\rho)|-|\sym(\varphi)|}\prod_{c\in T_\rho\cap\symc(\varphi)}\omega(c),
$$
where $\omega(c)$ is the weight of one edge of the
symmetric cycle $c$.
\begin{rem}
The multiplicity introduced here equals the one in
\cite[Definition $3.3$]{rau2019}. In \cite[Definition $3.3$]{rau2019}, the symmetric cycle $c$ was allowed to
be chosen in $\symc(\varphi)\setminus T_\rho$.
We compensate the contribution of even cycles in $\symc(\varphi)\setminus T_\rho$ by multiplying $2^{|E(\sym(\varphi))|}$ by a factor:
$$
2^{|E(T_\rho)|}=2^{|E(\sym(\varphi))|+2k},
$$
where $k$ is the number of even symmetric cycles in
$\symc(\varphi)\setminus T_\rho$.
\end{rem}

Let $(\varphi,\rho)$ be a real tropical cover.
A inner vertex $x_i\in\undl x$ is called a
{\it positive} or {\it negative} point if
it is the image of a $3$-valent vertex of $C$
which is depicted in Figure $\ref{fig:pt-positive}$ or
Figure $\ref{fig:pt-negative}$,
respectively, up to reflection along a vertical line.
We denote by $\undl x^+$ and $\undl x^-$ the collection
of positive and negative points in $\undl x$ respectively.
Note that a colouring $\rho(\varphi)$ of a tropical cover $\varphi$
induces a splitting of $\undl x=\undl x^+\sqcup\undl x^-$
into positive and negative branch points.
\begin{figure}[h]
    \centering
    \begin{tikzpicture}
    \draw[line width=0.4mm] (-3,0)--(-2,0)--(-1,0.5);
    \draw[line width=0.4mm,blue] (-2,0)--(-1,-0.5);
    \draw[line width=0.4mm,blue] (0,0)--(1,0)--(2,0.5);
    \draw[line width=0.4mm,blue] (1,0)--(2,-0.5);
    \draw[line width=0.4mm,red] (3,0)--(4,0);
    \draw[line width=0.4mm] (5,0.5)--(4,0)--(5,-0.5);
    \draw[line width=0.4mm,blue] (6,0)--(7,0);
    \draw[line width=0.4mm,dotted] (8,0.5)--(7,0)--(8,-0.5);
    \end{tikzpicture}
    \caption{The four types of positive vertices: even edges are drawn in colours, odd edges in black. Dotted edges are
    the symmetric cycles or forks contained in $T_I$.}
    \label{fig:pt-positive}
\end{figure}

\begin{figure}[h]
    \centering
    \begin{tikzpicture}
    \draw[line width=0.4mm] (-3,0)--(-2,0)--(-1,0.5);
    \draw[line width=0.4mm,red] (-2,0)--(-1,-0.5);
    \draw[line width=0.4mm,red] (0,0)--(1,0)--(2,0.5);
    \draw[line width=0.4mm,red] (1,0)--(2,-0.5);
    \draw[line width=0.4mm,blue] (3,0)--(4,0);
    \draw[line width=0.4mm] (5,0.5)--(4,0)--(5,-0.5);
    \draw[line width=0.4mm,red] (6,0)--(7,0);
    \draw[line width=0.4mm,dotted] (8,0.5)--(7,0)--(8,-0.5);
    \end{tikzpicture}
    \caption{The four types of negative vertices.}
    \label{fig:pt-negative}
\end{figure}

\begin{thm}[{\cite[Corollary $5.9$]{mr-2015}}]
\label{thm:real-tropical-Hurwitz}
Let $d,g,\lambda,\mu$ and $\undl x\subset\rb$ be as above.
Suppose $\undl x=\undl x^+\sqcup\undl x^-$ is a splitting
such that $|\undl x^+|=s$. Then
$$
H^\rb_g(\lambda,\mu;s)=\sum_{[(\varphi,\rho)]}\mult^\rb(\varphi,\rho),
$$
where we sum over all isomorphism classes $[(\varphi,\rho)]$
of real tropical covers of type $(g,\lambda,\mu,\undl x)$ whose positive and negative branch points reproduce the splitting $\undl x^+,\undl x^-$.
\end{thm}
\begin{rem}
We use the real multiplicity introduced in \cite{rau2019}
which differs from the one in \cite{mr-2015}.
The subset $T_\rho$ of a colouring $\rho$
of a tropical cover $\varphi: C\to T\pb^1$ was allowed to
contain even symmetric forks in \cite{mr-2015}.
The readers may refer to \cite[Remark $3.5$]{rau2019}
for more detail analysis on the difference between
these two definitions.
\end{rem}

\section{Zigzag covers and effective non-zigzag covers}
From Proposition $\ref{prop:rHurwitz-symmetry}$, we know that
if $H$ is a lower bound for
$H^\rb_g(\lambda,\mu;s)$,
$\left \lceil\frac{r}{2}\right\rceil\leq s\leq r$,
$H$ is also a lower bound
for $H^\rb_g(\lambda,\mu;s)$, $0\leq s\leq r$.
In the rest of this paper,
we find some tropical covers
with even weight which contribute to
$H^\rb_g(\lambda,\mu;s)$ for
$\left \lceil\frac{r}{2}\right\rceil\leq s\leq r$.

Let us recall the lower bound established in \cite{rau2019} and the properties of it.
A {\it string} $S$ in a tropical curve $C$ is a connected
subgraph such that $S\cap C^\circ$ is a closed submanifold
of $C^\circ$.
\begin{figure}[h]
    \centering
    \begin{tikzpicture}
    \draw[line width=0.4mm] (-3,0)--(-1.5,0);
    \draw[line width=0.4mm] (-0.3,0) arc[start angle=0, end angle=360, x radius=0.6, y radius=0.4];
    \draw[line width=0.4mm] (-0.3,0)--(1,0);
    \draw[line width=0.4mm] (2.2,0) arc[start angle=0, end angle=360, x radius=0.6, y radius=0.4];
    \draw[line width=0.4mm] (2.2,0)--(3.7,0);
    \draw[line width=0.4mm,gray]  (4.2,0.3)--(3.7, 0) -- (4.2,-0.3);
    \draw[line width=0.4mm,gray] (4.5,0) node{$S$};
    \draw (-3.5,0) node{$2o$};
    \draw[line width=0.4mm]  (-3,1.8)--(-2.5, 1.5) -- (-3,1.2);
    \draw[line width=0.4mm] (-2.5,1.5)--(-1.5,1.5);
    \draw[line width=0.4mm] (-0.3,1.5) arc[start angle=0, end angle=360, x radius=0.6, y radius=0.4];
    \draw[line width=0.4mm] (-0.3,1.5)--(1,1.5);
    \draw[line width=0.4mm] (2.2,1.5) arc[start angle=0, end angle=360, x radius=0.6, y radius=0.4];
    \draw[line width=0.4mm] (2.2,1.5)--(3.7,1.5);
    \draw[line width=0.4mm,gray]  (4.2,1.8)--(3.7, 1.5) -- (4.2,1.2);
    \draw[line width=0.4mm] (-3.5,1.8) node{$o$};
    \draw[line width=0.4mm] (-3.5,1.2) node{$o$};
    \draw[line width=0.4mm,gray] (4.5,1.5) node{$S$};
    \draw[line width=0.4mm]  (-3,-1.5)--(3.7, -1.5);
    \draw[line width=0.4mm,gray]  (4.2,-1.8)--(3.7, -1.5) -- (4.2,-1.2);
    \draw[line width=0.4mm] (-3.5,-1.5) node{$2e$};
    \draw[line width=0.4mm,gray] (4.5,-1.5) node{$S$};
    \end{tikzpicture}
    \caption{Tails for zigzag covers. It does not matter
    whether $S$ turns or not here. The number of cycles
    in the first two types can be arbitrary.}
    \label{fig:zigzag}
\end{figure}

\begin{defi}[{\cite[Definition 4.4]{rau2019}}]
\label{def:zigzag}
A zigzag cover is a tropical cover $\varphi:C\to T\pb^1$
if there is a subset $S\subset C\setminus\sym(\varphi)$
satisfying
\begin{itemize}
    \item $S$ is either a string of odd edges or consists
    of a single inner vertex;
    \item the connected components of $C\setminus S$ are
    of the type depicted in Figure $\ref{fig:zigzag}$. In Figure $\ref{fig:zigzag}$, all the cycles and forks
    are symmetric and of odd weight.
\end{itemize}
\end{defi}

\begin{rem}
For the convenience, in Figure $\ref{fig:zigzag}$ and in
the rest of this paper, we use the following
notations: the variables
$o,o_1,o_2,\ldots$ are used to denote odd integers,
and $e,e_1,e_2,\ldots$ are used to denote even integers.
\end{rem}

\begin{lem}[{\cite[Lemma $4.3$]{rau2019}}]\label{lem:parity}
For any real tropical cover $(\varphi,\rho)$ the multiplicity
$\mult^\rb(\varphi,\rho)$ is an integer
whose parity is independent of the colouring $\rho$.
\end{lem}

\begin{prop}[{\cite[Proposition $4.7$]{rau2019}}]\label{prop:zigzag-parity}
The real tropical cover $(\varphi,\rho)$ is of odd multiplicity
if and only if $\varphi$ is a zigzag cover.
\end{prop}

The tropical covers with even weight which we are looking
for are characterized by the following definition.
\begin{defi}\label{def:non-zigzag}
An effective non-zigzag cover of type
$(g,\lambda,\mu,\undl x)$
is a tropical cover $\varphi:C\to T\pb^1$
of type $(g,\lambda,\mu,\undl x)$
satisfying the following conditions:
\begin{itemize}
    \item there are $n$ strings $S_1,\ldots,S_n
    \subset C\setminus\sym(\varphi)$ of odd edges,
    $n>1$.
    \item the connected components of
    $C\setminus (\cup_{i=1}^{n} S_i)$ are of the type depicted in Figure $\ref{fig:zigzag}$ and Figure $\ref{fig:non-zigzag}$.
    In Figure $\ref{fig:zigzag}$,
    all the cycles and forks
    are symmetric of odd weight. In
    Figure $\ref{fig:non-zigzag}$, two strings are connected by exactly one bridge edge.
    \item
    Every inner vertex of the
    bridge edges depicted in Figure
    $\ref{fig:non-zigzag}$
    is mapped to some
    $x_j$ by $\varphi$ with
    $j\leq\left \lceil\frac{r}{2}\right\rceil$.
\end{itemize}
\end{defi}

\begin{figure}[h]
    \centering
    \begin{tikzpicture}
    \draw[line width=0.4mm,gray]  (-3.5,0.5)--(-3, 0.5) -- (-2.5,0.2);
    \draw[line width=0.4mm] (-3,0.5)--(2,0.5);
    \draw[line width=0.4mm,gray]  (1.5,0.2)--(2,0.5) -- (2.5,0.5);
    \draw[line width=0.4mm,gray] (3,0.5) node{$S_{i+1}$};
    \draw[line width=0.4mm,gray] (-4,0.5) node{$S_{i}$};
    \draw[line width=0.4mm,gray] (0.5,0.7) node{$E_{i}$};
    \draw[line width=0.4mm,gray]  (-3.5,-0.7)--(-3,-1) -- (-3.5,-1.3);
    \draw[line width=0.4mm] (-3,-1)--(2,-1);
    \draw[line width=0.4mm,gray]  (2.5,-1.3)--(2,-1) -- (2.5,-0.7);
    \draw[line width=0.4mm,gray] (3,-1) node{$S_{i+1}$};
    \draw[line width=0.4mm,gray] (-4,-1) node{$S_i$};
    \draw[line width=0.4mm,gray] (-0.5,-0.8) node{$E_{i}$};
    \end{tikzpicture}
    \caption{Bridge edges connecting two strings.
    The bending behaviour of the strings matters here.}
    \label{fig:non-zigzag}
\end{figure}

\begin{rem}\label{rem:non-zigzag}
From Lemma $\ref{lem:parity}$ and Proposition
$\ref{prop:zigzag-parity}$, we know that the multiplicity
of an effective non-zigzag cover is a positive even integer.
\end{rem}

\begin{prop}
Let $\varphi$ be an effective non-zigzag cover simply branched at $\undl x$.
If $\undl x=\undl x^+\sqcup\undl x^-$ is a splitting
such that $\{x_1,...,x_{\left \lceil\frac{r}{2}\right\rceil}\}\subset\undl x^+$.
Then there is a unique colouring $\rho$ of $\varphi$ such that the real
tropical cover $(\varphi,\rho)$ has
positive and negative branch points
as the splitting $\undl x=\undl x^+\sqcup\undl x^-$.
\end{prop}

\begin{proof}
Suppose that $v$ is an
inner vertex of $S_i$, $i=2,\ldots,n$.
If the even edge $E_i$ adjacent to $v$
is the first type depicted
in Figure $\ref{fig:non-zigzag}$, the colouring of $E_i$
is blue. Otherwise, the colouring of $E_i$
is red.

Assume that $v_1\in S_1$ is a vertex from which a given
tail $C_1$ depicted in Figure $\ref{fig:zigzag}$ emanates.
The colour rules from Figure $\ref{fig:pt-positive}$
and Figure $\ref{fig:pt-negative}$ induces
a unique colouring of even edges of this tail.
The method to impose colouring is described in
the proof of \cite[Proposition $4.8$]{rau2019},
so we omit it here.
\end{proof}

The \textit{zigzag number} $Z_g(\lambda,\mu)$ is the number of
zigzag covers of type $(g,\lambda,\mu,\undl x)$ \cite[Definition $4.9$]{rau2019}.

\begin{defi}
\label{def:effective-number}
The effective non-zigzag number $Z'_g(\lambda,\mu)$ is twice the
number of effective non-zigzag covers of
type $(g,\lambda,\mu,\undl x)$, and the sum
$E_g(\lambda,\mu)=Z_g(\lambda,\mu)+Z_g'(\lambda,\mu)$
is called the effective number.
\end{defi}

From Remark $\ref{rem:non-zigzag}$, the multiplicity of an
effective non-zigzag cover is at least $2$.
Therefore, we take twice the number of effective non-zigzag covers as the effective non-zigzag number.
\begin{rem}
In \cite[Remark $5.3$]{rau2019}, a combinatorial description of tropical covers was given.
Since we assume the points
$\undl x$ satisfying $x_1<\ldots<x_r$, the fourth condition
of Definition $\ref{def:non-zigzag}$ is actually a condition on the
total order of inner vertices of $C$. Let $t_C$ be the number of total orders, satisfying the fourth condition
of Definition $\ref{def:non-zigzag}$ and compatible with the partial order induced by the orientation,
of the inner vertices of $C$.
It follows from \cite[Remark $5.3$]{rau2019} that $Z'_g(\lambda,\mu)$ depends
on the number of tropical curves $C$ satisfying the first three conditions of
Definition $\ref{def:non-zigzag}$ and $t_C$.
So $Z'_g(\lambda,\mu)$ does not depend on the choice of  $\undl x\subset\rb$. Hence the effective number
$E_g(\lambda,\mu)$ does not depend
on $\undl x$.
\end{rem}

\begin{prop}\label{thm:lower-bound1}
Fix an integer $g\geq0$, and two partitions
$\lambda$, $\mu$ such that $|\lambda|=|\mu|$,
$\{\lambda,\mu\}\not\subset
\{(2k),(k,k)\}$ and $r=l(\lambda)+l(\mu)+2g-2>0$.
Then the number of
real ramified covers is bounded from
below by the effective number and
they have the same parity:
\begin{align*}
    Z_g&(\lambda,\mu)\leq E_g(\lambda,\mu)\leq H^\rb_g(\lambda,\mu;s)\leq
    H^\cb_g(\lambda,\mu),\\
    Z_g(\lambda,\mu&)\equiv E_g(\lambda,\mu)\equiv H^\rb_g(\lambda,\mu;s)\equiv
    H^\cb_g(\lambda,\mu)\mod 2.
\end{align*}
\end{prop}

\begin{proof}
It is straightforward
from \cite[Theorem $4.10$]{rau2019} and Definition
$\ref{def:effective-number}$.
\end{proof}

\begin{rem}
Note that $\{\lambda,\mu\}\not\subset
\{(2k),(k,k)\}$ and $l(\lambda)+l(\mu)+2g-2>0$ is
a sufficient condition to guarantee that
real and complex double Hurwitz numbers
are actual counts of covers
(See \cite[Remark $2.5$, Remark $2.6$]{rau2019}
for more details).
\end{rem}

We end this subsection by providing some examples of zigzag covers and effective non-zigzag covers.
\begin{exa}
We present two different kinds of zigzag covers in Figure $\ref{fig:2-zigzag}$,
\begin{figure}[h]
    \centering
    \begin{tikzpicture}
    \draw[line width=0.3mm,gray] (-2,-0.4)--(-1,-0.7)--(0,-1)--
    (-2,-1.3);
    \draw[line width=0.3mm,gray] (-2,-1)--(-1,-0.7);
    \draw[line width=0.3mm,gray] (0,-1)--(1,-1);
    \draw[line width=0.3mm,gray] (1.8,-1) arc[start angle=0, end angle=360, x radius=0.4, y radius=0.2];
    \draw[line width=0.3mm,gray] (1.8,-1)--(2.5,-1)--(3,-0.7);
    \draw[line width=0.3mm,gray] (3,-1.3)--(2.5,-1);
    \draw[fill=black] (0,-1) circle (0.05);
    \draw (0,-0.8) node{$S$};
    \draw[line width=0.3mm,gray] (-2,-1.8)--(3,-1.8);
    \foreach \Point in {(-2,-1.8),(-1,-1.8),(0,-1.8), (1,-1.8),(1.8,-1.8),(2.5,-1.8),(3,-1.8)}
    \draw[fill=black] \Point circle (0.05);
    \draw (-2,-2) node{$-\infty$} (-1,-2) node{$x_1$} (0,-2) node{$x_2$}  (1,-2) node{$x_3$} (1.8,-2) node{$x_4$} (2.5,-2)node{$x_5$} (3,-2) node{$\infty$};
    \draw (0.7,-3) node{$(1)$ $S$ is a point};

    \draw[line width=0.4mm] (4,-0.4)--(7.5,-0.8)--(6.3,-1.2)--
    (9,-1.5);
    \draw[line width=0.3mm,gray] (7.5,-0.8)--(9,-0.8);
    \draw[line width=0.3mm,gray] (4,-1)--(4.6,-1.2)--(5,-1.2);
    \draw[line width=0.3mm,gray] (5.8,-1.2) arc[start angle=0, end angle=360, x radius=0.4, y radius=0.2];
    \draw[line width=0.3mm,gray] (5.8,-1.2)--(6.3,-1.2);
    \draw[line width=0.3mm,gray] (4,-1.4)--(4.6,-1.2);
    \draw[line width=0.3mm,gray] (4,-1.8)--(9,-1.8);
    \foreach \Point in {(4,-1.8),(4.6,-1.8),(5,-1.8), (5.8,-1.8),(6.3,-1.8),(7.5,-1.8),(9,-1.8)}
    \draw[fill=black] \Point circle (0.05);
    \draw (6.8,-0.5) node{$S$};
    \draw (4,-2) node{$-\infty$} (4.6,-2) node{$x_1$} (5,-2) node{$x_2$}  (5.8,-2) node{$x_3$} (6.3,-2) node{$x_4$} (7.5,-2)node{$x_5$} (9,-2) node{$\infty$};
    \draw (7,-3) node{$(2)$ $S$ is a string};
    \end{tikzpicture}
    \caption{Two zigzag covers.}
    \label{fig:2-zigzag}
\end{figure}
two types of effective non-zigzag covers in Figure $\ref{fig:2-non-zigzag}$,
\begin{figure}[h]
    \centering
    \begin{tikzpicture}
    \draw[line width=0.4mm] (-2,-0.4)--(-0.8,-0.7)--(3,-0.7);
    \draw[line width=0.3mm,gray] (-1.2,-1)--(-0.8,-0.7);
    \draw[line width=0.4mm] (-2,-1)--(-1.2,-1)--(1,-1.6)--(0,-1.8)--(3,-2.1);
    \draw[line width=0.3mm,gray] (-2,-1.8)--(0,-1.8);
    \draw[line width=0.3mm,gray] (1,-1.6)--(2,-1.6)--(3,-1.3);
    \draw[line width=0.3mm,gray] (2,-1.6)--(3,-1.9);
    \draw[line width=0.3mm,gray] (-2,-2.3)--(3,-2.3);
    \foreach \Point in {(-2,-2.3),(-1.2,-2.3),(-0.8,-2.3), (0,-2.3),(1,-2.3),(2,-2.3),(3,-2.3)}
    \draw[fill=black] \Point circle (0.05);
    \draw (-0.8,-0.5) node{$S_2$}  (-1.35,-0.8) node{$E_1$} (0.2,-1.2) node{$S_1$};
    \draw (0.7,-3.2) node{$(1)$ $E_1$ is of type $1$};
    \draw (-2,-2.5) node{$-\infty$} (-1.2,-2.5) node{$x_1$} (-0.8,-2.5) node{$x_2$}  (0,-2.5) node{$x_3$} (1,-2.5) node{$x_4$} (2,-2.5)node{$x_5$} (3,-2.5) node{$\infty$};

    \draw[line width=0.4mm] (4,-0.4)--(4.6,-0.6)--(4,-0.8);
    \draw[line width=0.3mm,gray] (4.6,-0.6)--(5.5,-0.6);
    \draw[line width=0.4mm] (9,-0.4)--(5.5,-0.6)--(8,-1.2)--(6.6,-1.6)--(9,-1.9);
    \draw[line width=0.3mm,gray] (4,-1.3)--(5.9,-1.6)--(6.6,-1.6);
    \draw[line width=0.3mm,gray] (4,-1.9)--(5.9,-1.6);
    \draw[line width=0.3mm,gray] (8,-1.2)--(9,-1.2);
    \draw[line width=0.3mm,gray] (4,-2.3)--(9,-2.3);
    \foreach \Point in {(4,-2.3),(4.6,-2.3),(5.5,-2.3), (5.9,-2.3),(6.6,-2.3),(8,-2.3),(9,-2.3)}
    \draw[fill=black] \Point circle (0.05);
    \draw (4.3,-0.3) node{$S_2$}  (4.9,-0.4) node{$E_1$} (6,-0.3) node{$S_1$};
    \draw (7,-3.2) node{$(2)$ $E_1$ is of type $2$};
    \draw (4,-2.5) node{$-\infty$} (4.6,-2.5) node{$x_1$} (5.5,-2.5) node{$x_2$}  (5.9,-2.5) node{$x_3$} (6.6,-2.5) node{$x_4$} (8,-2.5)node{$x_5$} (9,-2.5) node{$\infty$};
    \end{tikzpicture}
    \caption{Two effective non-zigzag covers.}
    \label{fig:2-non-zigzag}
\end{figure}
and a tropical cover that is not zigzag, nor effective non-zigzag cover in
Figure $\ref{fig:not-zigzags}$.
\begin{figure}[h]
    \centering
    \begin{tikzpicture}
    \draw[line width=0.4mm] (4,-0.4)--(7,-0.8)--(4,-1.2);
    \draw[line width=0.3mm,gray] (7,-0.8)--(8,-0.8);
    \draw[line width=0.4mm] (11,-0.4)--(8,-0.8)--(10,-1.0)--(6.6,-1.6)--(11,-1.9);
    \draw[line width=0.3mm,gray] (4,-1.35)--(5.9,-1.6)--(6.6,-1.6);
    \draw[line width=0.3mm,gray] (4,-1.85)--(5.9,-1.6);
    \draw[line width=0.3mm,gray] (10,-1)--(11,-1);
    \draw[line width=0.3mm,gray] (4,-2.3)--(11,-2.3);
    \foreach \Point in {(4,-2.3),(6.6,-2.3),(7,-2.3), (8,-2.3),(5.9,-2.3),(10,-2.3),(11,-2.3)}
    \draw[fill=black] \Point circle (0.05);
    \draw (4.9,-0.3) node{$S_2$}  (7.5,-0.5) node{$E_1$} (10,-0.3) node{$S_1$};
    \draw (4,-2.5) node{$-\infty$} (5.9,-2.5) node{$x_1$} (6.6,-2.5) node{$x_2$}  (7,-2.5) node{$x_3$} (8,-2.5) node{$x_4$} (10,-2.5)node{$x_5$} (11,-2.5) node{$\infty$};
    \end{tikzpicture}
    \caption{A tropical cover that is not zigzag, nor effective
    non-zigzag.}
    \label{fig:not-zigzags}
\end{figure}
\end{exa}

\section{Asymptotic behaviour of effective non-zigzag covers}
In this section, we follow closely the arguments of \cite[Section $5$]{rau2019} to investigate the asymptotic behaviour of effective non-zigzag covers.
A sufficient condition
to guarantee the existence of effective non-zigzag covers
is given, and the asymptotic behaviour of effective numbers is also
considered.

We first recall the definition of tail decomposition
of a partition $\lambda$ introduced in \cite{rau2019}.
Let $\lambda=(\lambda_1,\ldots,\lambda_n)$ and
$\mu=(\mu_1,\ldots,\mu_m)$ be two partitions of $d$.
We use $\lambda_i\in\lambda$ to mean that
$\lambda_i$ is a part of $\lambda$.
We denote
\begin{align*}
2\lambda&:=(2\lambda_1,2\lambda_2,\ldots,2\lambda_n),\\
\lambda^2&:=(\lambda_1,\lambda_1,\ldots,\lambda_n,\lambda_n),\\
(\lambda,\mu)&:=(\lambda_1,\ldots,\lambda_n,\mu_1,\ldots,\mu_m).
\end{align*}
For any partition $\lambda$, there is a unique
decomposition $\lambda=(2\lambda_{2e},2\lambda_{2o},
\lambda_{o,o}^2,\lambda_o)$
such that
\begin{itemize}
    \item every part $\lambda_{2e}^i$ in $\lambda_{2e}$ is an even number;
    \item every part $\lambda_{2o}^i$ or $\lambda_{o,o}^i$ in $\lambda_{2o}$ or $\lambda_{o,o}$ respectively is an odd number;
    \item every part $\lambda_o^i$ in $\lambda_o$ is an
    odd number and the odd number $\lambda_o^i$ does
    not appear more than once in $\lambda_o$.
\end{itemize}
Such a decomposition is called the {\it tail decomposition} of $\lambda$.
Let $l(\lambda_{1,1})$ denote the number
of ones which appear in $\lambda_{o,o}$.

\begin{exa}
Assume $\lambda=(7,6,4,5,5,3,1,1,1)$, then the tail
decomposition of $\lambda$ is
$\lambda=(2\lambda_{2e},2\lambda_{2o},
\lambda_{o,o}^2,\lambda_o)$,
where $\lambda_{2e}=(2)$, $\lambda_{2o}=(3)$,
$\lambda_{o,o}=(5,1)$, $\lambda_o=(7,3,1)$.
\end{exa}

\begin{defi}\label{def:pmenz}
An effective non-zigzag cover
$\varphi:C\to T\pb^1$ of type
$(g,\lambda,\mu,\undl x)$
is properly mixed if all tails of the type depicted
in Figure $\ref{fig:zigzag}$ are attached to
a unique string $S$ such that
the branch points $x_1<x_2<\ldots<x_r$
grouped in subsets of respective sizes $f$, $t_l$, $b_l$,
$u-1$, $b_r$, $2g$, $t_r$
occur as images of
\begin{itemize}
    \item the vertices of the bridge edges of $C$;
    \item the symmetric
    fork vertices of tails of type $o,o$ labelled by $\lambda$ attaching to $S$;
    \item bends of $S$ with peaks pointing to the left;
    \item unbent vertices of $S$ other than the intersection point
    of the bridge edge $E$ and the string $S$;
    \item bends of $S$ with peaks pointing to the right;
    \item the vertices of symmetric cycles located on
    tails labelled by $\mu$;
    \item the symmetric fork vertices of tails of type
    $o,o$ labelled by $\mu$ attaching to $S$.
\end{itemize}
Here, $f$ is the number of vertices of the bridge edges
and $f\leq \left\lceil\frac{r}{2}\right\rceil$,
$t_l$ or $t_r$ is the number of tails of
the first type depicted in Figure $\ref{fig:zigzag}$
labelled by $\lambda_{o,o}$ or $\mu_{o,o}$
respectively,
$b_l$ or $b_r$ is the number
of bends of the string $S$ with the peaks pointing to
the left or right respectively,
and $u$ is the number of unbent vertices of $S$.
\end{defi}

\begin{exa}
The effective non-zigzag cover depicted in Figure $\ref{fig:effective-nonzigzag4}$ is properly mixed.
In this effective non-zigzag cover, $f=6$.
\begin{figure}[h]
    \centering
    \begin{tikzpicture}
    \draw[line width=0.4mm] (-4,2.4)--(-2.2,2.2)--(-2.5,2)--(-4,1.8);
    \draw[line width=0.3mm,gray] (-2.2,2.2)--(-1.5,2.2);
    \draw[line width=0.4mm] (7,2.6)--(-1.5,2.2)--(7,1.8);
    \draw[line width=0.3mm,gray] (-2.5,2)--(-2.8,1.7);
    \draw[line width=0.4mm] (-4,1.5)--(-3.1,1.6)--(-2.8,1.7)--(7,1.6);
    \draw[line width=0.3mm,gray] (-3.1,1.6)--(-3.2,1.2);
    \draw[line width=0.4mm] (-4,0.5)-- (-3.2,1.2)--(3.2,1)--(4,0.8)--(2.9,0.6)--
    (1.8,0.4)--(3.1,0.32)--(3.8,0.2)--(3,0.15)--(1.9,0)--(4.1,-0.2)--(2,-0.4)--(7,-0.6);
    \draw[line width=0.3mm,gray] (3.2,1)--(4.2,1);
    \draw[line width=0.3mm,gray] (3.1,0.32)--(4.2,0.32);
    \draw[line width=0.3mm,gray] (2.9,0.6)--(1.7,0.6);
    \draw[line width=0.3mm,gray] (3,0.15)--(1.7,0.15);
    \draw[line width=0.3mm,gray] (-4,0.4)--(-1.3,0.4);
    \draw[line width=0.3mm,gray] (-4,-0.4)--(-1.3,-0.4);
    \draw[line width=0.3mm,gray] (-1.3,0)--(-4,0);
    \draw[line width=0.3mm,gray] (4,0.8)--(4.2,0.8);
    \draw[line width=0.3mm,gray] (1.8,0.4)--(1.7,0.4);
    \draw[line width=0.3mm,gray] (3.8,0.2)--(4.2,0.2);
    \draw[line width=0.3mm,gray] (1.9,0)--(1.7,0);
    \draw[line width=0.3mm,gray] (4.1,-0.2)--(4.2,-0.2);
    \draw[line width=0.3mm,gray] (2,-0.4)--(1.7,-0.4);
    \draw[line width=0.3mm,gray,dotted] (4.2,0.8)--(4.5,0.8);
    \draw[line width=0.3mm,gray,dotted] (4.2,0.32)--(4.5,0.32);
    \draw[line width=0.3mm,gray,dotted] (4.2,1)--(4.5,1);
    \draw[line width=0.3mm,gray,dotted] (4.2,0.2)--(4.5,0.2);
    \draw[line width=0.3mm,gray,dotted] (4.2,-0.2)--(4.5,-0.2);
    \draw[line width=0.3mm,gray,dotted] (1.4,0.4)--(1.7,0.4);
    \draw[line width=0.3mm,gray,dotted] (1.4,0.15)--(1.7,0.15);
    \draw[line width=0.3mm,gray,dotted] (1.4,0.6)--(1.7,0.6);
    \draw[line width=0.3mm,gray,dotted] (1.4,-0.4)--(1.7,-0.4);
    \draw[line width=0.3mm,gray,dotted] (1.4,0)--(1.7,0);
    \draw[line width=0.3mm,gray,dotted] (-1.3,0.4)--(-1,0.4);
    \draw[line width=0.3mm,gray,dotted] (-1.3,-0.4)--(-1,-0.4);
    \draw[line width=0.3mm,gray,dotted] (-1.3,0)--(-1,0);
    \draw[line width=0.3mm,gray,dotted] (-1,0.2)--(-0.8,0.2);
    \draw[line width=0.3mm,gray,dotted] (-1,-0.2)--(-0.8,-0.2);
    \draw[line width=0.3mm,gray,dotted] (-1,-0.5)--(-0.8,-0.5);
    \draw[line width=0.3mm,gray,dotted] (-1,0.5)--(-0.8,0.5);
    \draw[line width=0.3mm,gray,dotted] (-1,0.1)--(-0.8,0.1);
    \draw[line width=0.3mm,gray,dotted] (-1,-0.1)--(-0.8,-0.1);
    \draw[line width=0.3mm,gray] (-0.8,0.1)--(0.4,0)--(-0.8,-0.1);
    \draw[line width=0.3mm,gray] (0.4,0)--(1.1,0);
    \draw[line width=0.3mm,gray,dotted] (1.1,0)--(1.4,0);
    \draw[line width=0.3mm,gray] (-0.8,-0.2)--(0.2,-0.35)--(-0.8,-0.5);
    \draw[line width=0.3mm,gray] (0.2,-0.35)--(1.1,-0.35);
    \draw[line width=0.3mm,gray,dotted] (1.1,-0.35)--(1.4,-0.35);
    \draw[line width=0.3mm,gray] (-0.8,0.2)--(0.3,0.35)--(-0.8,0.5);
    \draw[line width=0.3mm,gray] (0.3,0.35)--(1.1,0.35);
    \draw[line width=0.3mm,gray,dotted] (1.4,0.35)--(1.1,0.35);
    \draw[line width=0.3mm,gray] (5.4,0.6) arc[start angle=0, end angle=360, x radius=0.25, y radius=0.1];
    \draw[line width=0.3mm,gray,dotted]  (4.5,0.6)--(4.7,0.6);
    \draw[line width=0.3mm,gray] (4.7,0.6)--(4.9,0.6);
    \draw[line width=0.3mm,gray,dotted]  (5.6,0.6)--(5.8,0.6);
    \draw[line width=0.3mm,gray] (5.4,0.6)--(5.6,0.6);
    \draw[line width=0.3mm,gray] (5.2,-0.2) arc[start angle=0, end angle=360, x radius=0.25, y radius=0.1];
    \draw[line width=0.3mm,gray,dotted]  (4.4,-0.2)--(4.6,-0.2);
    \draw[line width=0.3mm,gray] (4.6,-0.2)--(4.7,-0.2);
    \draw[line width=0.3mm,gray,dotted]  (5.4,-0.2)--(5.8,-0.2);
    \draw[line width=0.3mm,gray] (5.2,-0.2)--(5.4,-0.2);
    \draw[line width=0.3mm,gray,dotted] (5.8,0.8)--(6,0.8);
    \draw[line width=0.3mm,gray] (6,0.8)--(6.5,0.8)--(7,1);
    \draw[line width=0.3mm,gray] (6.5,0.8)--(7,0.6);
    \draw[line width=0.3mm,gray,dotted] (5.8,0.2)--(6,0.2);
    \draw[line width=0.3mm,gray] (6,0.2)--(6.3,0.2)--(7,0.4);
    \draw[line width=0.3mm,gray] (6.3,0.2)--(7,0);
    \draw (6.4,-0.4) node{$S_4$};
    \draw (2.6,1.4) node{$S_3$};
    \draw (6,2.2) node{$S_1$};
    \draw (-3.3,2.5) node{$S_2$};
    \draw (-2.95,1.4) node{\tiny{$E_3$}};
    \draw (-2.45,1.84) node{\tiny{$E_2$}};
    \draw (-2,2.4) node{\tiny{$E_1$}};
    \draw (0.3,-1.3) node{\tiny{$t_l$}};
    \draw (1.9,-1.3) node{\tiny{$b_l$}};
    \draw (3,-1.3) node{\tiny{$u-1$}} (4,-1.3) node{\tiny{$b_r$}} (4.9,-1.3) node{\tiny{$2g$}} (6.4,-1.3) node{\tiny{$t_r$}};
    \draw[line width=0.3mm] (-4,-1)--(7,-1);
    \foreach \Point in {(-4,-1),(-3.2,-1),(-3.1,-1),(-2.8,-1), (-2.5,-1),(-2.2,-1), (-1.5,-1), (0.2,-1),(0.4,-1),(0.3,-1),(1.8,-1),(1.9,-1), (2,-1),(2.9,-1),(3,-1),(3.1,-1),(3.2,-1),(3.8,-1),(4,-1),(4.1,-1),(4.7,-1),(5.2,-1),(4.9,-1), (5.4,-1),(6.3,-1),(6.5,-1),(7,-1)}
    \draw[fill=black] \Point circle (0.03);
    \end{tikzpicture}
    \caption{An example of properly mixed effective non-zigzag cover.}
    \label{fig:effective-nonzigzag4}
\end{figure}
\end{exa}

\begin{prop}\label{prop:non-zigzag1}
Fix $g\in\nb$, partitions $\lambda$, $\mu$ with
$|\lambda|=|\mu|$.
Suppose that $l(\lambda_o,\mu_o)>2$ and $\sum_{i=1}^{l(\lambda_o)}\lambda_o^i\geq\sum_{i=1}^{l(\mu_o)}\mu_o^i$.
\begin{enumerate}
    \item[$(1)$] For sufficiently large even integer $m$,
    properly mixed effective non-zigzag covers $\varphi:C\to T\pb^1$ of type $(g,(\lambda,1^m),(\mu,1^m),\undl x)$ exist as in Definition $\ref{def:pmenz}$, where $f=2\max(l(\lambda_o),l(\mu_o))-2$.
    \item[$(2)$] For a properly mixed effective non-zigzag cover $\varphi:C\to T\pb^1$ obtained above,
    there exist constants $N_1$ and $N_2$ which are not dependent on $m$ such that
    $$
    \begin{aligned}
        t_l\geq\frac{m}{2}-N_1,~ t_r\geq\frac{m}{2}-N_1,\\
        b_l\geq\frac{m}{2}-N_2,~ b_r\geq\frac{m}{2}-N_2.
    \end{aligned}
    $$
\end{enumerate}
\end{prop}

\begin{proof}
$(1)$
Since $l(\lambda_o,\mu_o)\equiv|\lambda|+|\mu|=2d\mod 2$,
the number $l(\lambda_o,\mu_o)$ is even.
Suppose that $l(\lambda_o)-l(\mu_o)=2\ak$.
We construct a properly mixed zigzag cover of type $(g,(\lambda,1^m),(\mu,1^m),\undl x)$ by considering the three cases: $\ak=0$, $\ak>0$ and $\ak<0$.

Case $\ak\geq0$:
We assume that $m$ is a sufficiently large even integer,
so when $\ak>0$ it is possible to take $2\ak$ ones from $(1^m)$ and add them to the partition $\mu_o$.
We consider the partitions $\lambda_o$ and $(\mu_o,1^{2\ak})$ which have the same length $n=l(\lambda_o)$.
Since the partition $\lambda_o$ does not contain repeated elements,
we can find an ordering of the elements $\lambda_1,\ldots,\lambda_n$ of $\lambda_o$ and $\mu_1,\ldots,\mu_n$ of $(\mu_o,1^{2\ak})$ such that $w_k:=\sum_{i=1}^k(\lambda_i-\mu_i)$ is different from zero for all $k=1,\ldots,n-1$, and $w=\sum_{i=1}^n\lambda_i-\sum_{i=1}^{n-1}\mu_i>0$.
The existence of such an ordering is guaranteed by the following reasons.
Because there are always at least two distinct elements in $\lambda_o$ to choose from, it is possible to choose an ordering which guarantees that $w_k$ is non-zero for all $k=1,\ldots,n-1$.
Moreover, once the maximal odd integer in $\mu_o$ is chosen to be $\mu_n$ in the ordering $\mu_1,\ldots,\mu_n$, we obtain that
$w>0$ from the assumption that $\sum_{i=1}^{l(\lambda_o)}\lambda_o^i\geq\sum_{i=1}^{l(\mu_o)}\mu_o^i$.

We construct a weighted and oriented graph $C$ by the following recursive procedure:
\begin{itemize}
    \item[$(\textrm{i})$] We start with a string $S_1$ with two ends which are labelled by $\lambda_1$ and $\mu_1$.
    The rule of equipping orientation on an end is that any end labelled by a part of $(\lambda,1^m)$ (resp. $(\mu,1^m)$) is oriented inwards (resp. outwards).
    Therefore, the string $S_1$ is oriented from left to right.
    \item[$(\textrm{ii})$] We add a bridge edge $E_k$ connecting $S_k$ to the next string $S_{k+1}$, where $k=1,\ldots,n-1$. Suppose that $E_k$ intersects with $S_k$ at $u_k$ and intersects with $S_{k+1}$ at $v_{k+1}$, where $k=1,\ldots,n-1$. The weight of $E_k$ is $|w_k|$.
    If $w_k>0$, the bridge edge $E_k$ is oriented such that it points
    from $u_k$ to $v_{k+1}$ ({\it i.e.} $E_k$ directs to the right when coming from $S_k$).
    Otherwise, $E_k$ is oriented such that it points
    from $v_{k+1}$ to $u_k$ ({\it i.e.} $E_k$ directs to the left when coming from $S_k$).
    \item[$(\textrm{iii})$] The next string $S_{k+1}$ has two ends which are labelled by $\lambda_{k+1}$ and $\mu_{k+1}$, $k=1,\ldots,n-1$.
    Now we choose the position of $u_{k+1}$.
    If $w_{k}>0$, we choose $u_{k+1}$ on the right of $v_{k+1}$.
    Otherwise, we add $u_{k+1}$ on the left of $v_{k+1}$ (See Figure $\ref{fig:position-u}$ for a local picture of the bridge vertices $v_{k+1}$ and $u_{k+1}$).
    \begin{figure}[h]
        \centering
        \begin{tikzpicture}
        \draw[line width=0.4mm] (-4,2.2)--(-3.7,2)--(-0.5,2);
        \draw[line width=0.3mm,gray] (-3.7,2)--(-3.5,1.7);
        \draw[line width=0.3mm,gray,dotted] (-3.5,1.6)--(-3.5,1.3);
        \draw[line width=0.3mm,gray,dotted] (-2,1.6)--(-2,1.3);
        \draw[line width=0.3mm,gray,dotted] (-1,1.6)--(-1,1.3);
        \draw[line width=0.3mm,gray] (-3.4,1.3)--(-3,1);
        \draw[line width=0.4mm] (-4,0.9)--(-3,1)--(-1.5,1)--(-0.5,1.1);
        \draw[line width=0.3mm,gray] (-1.5,1)--(-0.7,0.7);
        \draw[line width=0.3mm,gray,dotted] (-0.6,0.7)--(-0.6,0.4);
        \draw (-2.8,0.8) node{\tiny{$v_{k+1}$}} (-1.6,0.8) node{\tiny{$u_{k+1}$}};
        \draw (-3.4,1.1) node{\tiny{$E_k$}} (-0.6,0.8) node{\tiny{$E_{k+1}$}} (-3.8,1.8) node{\tiny{$E_1$}} (-0.6,2.15) node{\tiny{$S_1$}} (-0.7,1.2) node{\tiny{$S_{k+1}$}};
        \draw (-2,0) node{$(1)$ $w_k>0$};
        \draw[line width=0.4mm] (4,2.2)--(3.7,2)--(0.5,2);
        \draw[line width=0.3mm,gray] (3.7,2)--(3.5,1.7);
        \draw[line width=0.3mm,gray,dotted] (2,1.6)--(2,1.3);
        \draw[line width=0.3mm,gray,dotted] (1,1.6)--(1,1.3);
        \draw[line width=0.3mm,gray,dotted] (3.5,1.6)--(3.5,1.3);
        \draw[line width=0.3mm,gray] (3.4,1.3)--(3,1);
        \draw[line width=0.4mm] (4,0.9)--(3,1)--(1.5,1)--(0.5,1.1);
        \draw[line width=0.3mm,gray] (1.5,1)--(0.7,0.7);
        \draw[line width=0.3mm,gray,dotted] (0.6,0.7)--(0.6,0.4);
        \draw (2.9,0.85) node{\tiny{$v_{k+1}$}} (1.6,0.8) node{\tiny{$u_{k+1}$}};
        \draw (3.4,1.1) node{\tiny{$E_k$}} (0.6,0.85) node{\tiny{$E_{k+1}$}} (3.8,1.8) node{\tiny{$E_1$}} (0.6,2.15) node{\tiny{$S_1$}} (0.7,1.2) node{\tiny{$S_{k+1}$}};
        \draw (2,0) node{$(2)$ $w_k<0$};
        \end{tikzpicture}
        \caption{Local picture of bridge vertices $v_{k+1}$ and $u_{k+1}$.}
        \label{fig:position-u}
    \end{figure}
    Note that this construction ensures that the internal bounded edge of $S_{k+1}$,
    where $k=1,\ldots,n-2$, carries a weight and is oriented from left to right by the balancing condition at the inner vertices of $S_{k+1}$.
    Hence all the strings $S_1,\ldots,S_{n-1}$ are oriented from left to right without bends.
    Moreover, all bridge edges are of the first type depicted in Figure $\ref{fig:non-zigzag}$.
    \item[$(\textrm{iv})$] We attach tails of the first, second or third type depicted in Figure $\ref{fig:zigzag}$ which correspond to each part of $((\lambda,1^m)_{o,o},(\mu,1^{m-2\ak})_{o,o})$, $(\lambda_{2o},\mu_{2o})$ or $(\lambda_{2e},\mu_{2e})$ respectively to the final string $S_n$ on the right of the final bridge vertex $v_n$.
    Note that $w=\sum_{i=1}^n\lambda_i-\sum_{i=1}^{n-1}\mu_i>0$, so attaching the tails on the right of the final bridge vertex $v_n$ is possible (See Figure $\ref{fig:local-picture-v}$ for a local picture of the bridge vertex $v_{n}$).
    \begin{figure}[h]
        \centering
        \begin{tikzpicture}
        \draw[line width=0.3mm,gray,dotted] (2,1.6)--(2,1.3);
        \draw[line width=0.3mm,gray,dotted] (3.5,1.6)--(3.5,1.3);
        \draw[line width=0.4mm] (0.5,1)--(1.5,1)--(3.7,0.7)--(2.5,0.4)--(6,0.1);
        \draw[line width=0.3mm,gray] (1.5,1)--(2,1.3);
        \draw[line width=0.3mm,gray] (3.7,0.7)--(4,0.7);
        \draw[line width=0.3mm,gray,dotted] (4,0.7)--(4.3,0.7);
        \draw[line width=0.3mm,gray,dotted] (5,0.7)--(5.3,0.7);
        \draw[line width=0.3mm,gray] (5.3,0.7)--(5.7,0.7)--(6,0.9);
        \draw[line width=0.3mm,gray] (5.7,0.7)--(6,0.5);
        \draw[line width=0.3mm,gray] (2.5,0.4)--(2.2,0.4)--(0.5,0.2);
        \draw[line width=0.3mm,gray] (0.5,0.6)--(2.2,0.4);
        \draw (1.6,0.8) node{\tiny{$v_{n}$}} (0.6,0.85) node{\tiny{$\lambda_{n}$}} (6,0.2) node{\tiny{$S_{n}$}} (1.4,1.2) node{\tiny{$E_{n-1}$}};
        \end{tikzpicture}
        \caption{Local picture of bridge vertex $v_{n}$ when $w_{n-1}<0$. Note that $\lambda_n>|w_{n-1}|$, where $|w_{n-1}|=\sum_{i=1}^{n-1}\mu_i-\sum_{i=1}^{n-1}\lambda_i$.}
        \label{fig:local-picture-v}
    \end{figure}
    At last, we obtain a graph $C$.
    Since $m$ is large enough, there is at least one tail of the first type depicted in Figure $\ref{fig:zigzag}$, on which we place $g$ balanced cycles.
    By the balancing condition, we extend the orientation and the weight function to all the edges of $C$.
    \item[$(\textrm{v})$] The orientation of edges of $C$ induces a partial order on the set of inner vertices.
    We now choose a total order of the inner vertices of the graph $C$ extending the partial order such that the set of bridge vertices $\{u_1, v_2,u_2,\ldots,v_{n-1},u_{n-1},v_n\}$ contains exactly the first $2n-2$ vertices.
    By our construction, it is clear that this is possible.
\end{itemize}
For sufficiently large even integer $m$, it is obvious that $2n-2<\frac{r}{2}=\frac{l(\lambda,1^m)+l(\mu,1^m)+2g-2}{2}$,
so the resulting tropical cover $\varphi:C\to T\pb^1$ is properly mixed.
See Figure $\ref{fig:effective-nonzigzag1}$ for an example in the case $\ak=0$,
and see Figure $\ref{fig:effective-nonzigzag2}$ for an example in the case $\ak>0$.
\begin{figure}[h]
    \centering
    \begin{tikzpicture}
    \draw[line width=0.4mm] (-4,2)--(-2.9,1.8)--(7,1.8);
    \draw[line width=0.3mm,gray] (-2.9,1.8)--(-2.5,1.5);
    \draw[line width=0.4mm] (-4,1.3)--(-2.5,1.5)--(7,1.5);
    \draw[line width=0.3mm,gray] (-2.2,1.5)--(-1.8,1.2);
    \draw[line width=0.4mm] (-4,0.5)--
    (-1.8,1.2);
    \draw[line width=0.4mm] (-1.8,1.2)--(4,0.8)--
    (1.8,0.4)--(3.8,0.2)--(1.9,0)--(4.1,-0.2)--(2,-0.4)--(7,-0.6);
    \draw[line width=0.3mm,gray] (-4,0.4)--(-1.3,0.4);
    \draw[line width=0.3mm,gray] (-4,-0.4)--(-1.3,-0.4);
    \draw[line width=0.3mm,gray] (-1.3,0)--(-4,0);
    \draw[line width=0.3mm,gray] (4,0.8)--(4.2,0.8);
    \draw[line width=0.3mm,gray] (1.8,0.4)--(1.7,0.4);
    \draw[line width=0.3mm,gray] (3.8,0.2)--(4.2,0.2);
    \draw[line width=0.3mm,gray] (1.9,0)--(1.7,0);
    \draw[line width=0.3mm,gray] (4.1,-0.2)--(4.2,-0.2);
    \draw[line width=0.3mm,gray] (2,-0.4)--(1.7,-0.4);
    \draw[line width=0.3mm,gray,dotted] (4.2,0.8)--(4.5,0.8);
    \draw[line width=0.3mm,gray,dotted] (4.2,0.2)--(4.5,0.2);
    \draw[line width=0.3mm,gray,dotted] (4.2,-0.2)--(4.5,-0.2);
    \draw[line width=0.3mm,gray,dotted] (1.4,0.4)--(1.7,0.4);
    \draw[line width=0.3mm,gray,dotted] (1.4,-0.4)--(1.7,-0.4);
    \draw[line width=0.3mm,gray,dotted] (1.4,0)--(1.7,0);
    \draw[line width=0.3mm,gray,dotted] (-1.3,0.4)--(-1,0.4);
    \draw[line width=0.3mm,gray,dotted] (-1.3,-0.4)--(-1,-0.4);
    \draw[line width=0.3mm,gray,dotted] (-1.3,0)--(-1,0);
    \draw[line width=0.3mm,gray,dotted] (-1,0.2)--(-0.8,0.2);
    \draw[line width=0.3mm,gray,dotted] (-1,-0.2)--(-0.8,-0.2);
    \draw[line width=0.3mm,gray,dotted] (-1,-0.5)--(-0.8,-0.5);
    \draw[line width=0.3mm,gray,dotted] (-1,0.5)--(-0.8,0.5);
    \draw[line width=0.3mm,gray,dotted] (-1,0.1)--(-0.8,0.1);
    \draw[line width=0.3mm,gray,dotted] (-1,-0.1)--(-0.8,-0.1);
    \draw[line width=0.3mm,gray] (-0.8,0.1)--(0.4,0)--(-0.8,-0.1);
    \draw[line width=0.3mm,gray] (0.4,0)--(1.1,0);
    \draw[line width=0.3mm,gray,dotted] (1.1,0)--(1.4,0);
    \draw[line width=0.3mm,gray] (-0.8,-0.2)--(0.2,-0.35)--(-0.8,-0.5);
    \draw[line width=0.3mm,gray] (0.2,-0.35)--(1.1,-0.35);
    \draw[line width=0.3mm,gray,dotted] (1.1,-0.35)--(1.4,-0.35);
    \draw[line width=0.3mm,gray] (-0.8,0.2)--(0.3,0.35)--(-0.8,0.5);
    \draw[line width=0.3mm,gray] (0.3,0.35)--(1.1,0.35);
    \draw[line width=0.3mm,gray,dotted] (1.4,0.35)--(1.1,0.35);
    \draw[line width=0.3mm,gray] (5.1,0.6) arc[start angle=0, end angle=360, x radius=0.15, y radius=0.1];
    \draw[line width=0.3mm,gray,dotted]  (4.5,0.6)--(4.7,0.6);
    \draw[line width=0.3mm,gray] (4.7,0.6)--(4.8,0.6);
    \draw[line width=0.3mm,gray] (5.1,0.6)--(5.2,0.6);
    \draw[line width=0.3mm,gray] (5.5,0.6) arc[start angle=0, end angle=360, x radius=0.15, y radius=0.1];
    \draw[line width=0.3mm,gray,dotted]  (5.6,0.6)--(5.8,0.6);
    \draw[line width=0.3mm,gray] (5.5,0.6)--(5.6,0.6);
    \draw[line width=0.3mm,gray,dotted] (5.8,0.8)--(6,0.8);
    \draw[line width=0.3mm,gray] (6,0.8)--(6.5,0.8)--(7,1);
    \draw[line width=0.3mm,gray] (6.5,0.8)--(7,0.6);
    \draw[line width=0.3mm,gray,dotted] (5.8,0.2)--(6,0.2);
    \draw[line width=0.3mm,gray] (6,0.2)--(6.3,0.2)--(7,0.4);
    \draw[line width=0.3mm,gray] (6.3,0.2)--(7,0);
    \draw (6.4,-0.4) node{$S_3$};
    \draw (2.6,1.3) node{$S_2$};
    \draw (2,2) node{$S_1$};
    \draw (-3.8,1.8) node{\tiny{$\lambda_1$}}
    (-3.8,1.2) node{\tiny{$\lambda_2$}} (-3.8,0.7) node{\tiny{$\lambda_3$}};
    \draw (6.8,1.9) node{\tiny{$\mu_1$}}
    (6.8,1.3) node{\tiny{$\mu_2$}} (6.8,-0.7) node{\tiny{$\mu_3$}};
    \draw (-1.8,1.35) node{\tiny{$E_2$}};
    \draw (-2.4,1.65) node{\tiny{$E_1$}};
    \draw[line width=0.3mm] (-4,-1)--(7,-1);
    \foreach \Point in {(-4,-1),(-2.5,-1), (-2.2,-1),(-1.8,-1),(-2.9,-1),(0.2,-1),(0.4,-1),(0.3,-1),(1.8,-1),(1.9,-1),(2,-1),(3.8,-1),(4,-1),(4.1,-1),(4.8,-1),(5.1,-1),(5.2,-1),(5.5,-1),(6.3,-1),(6.5,-1),(7,-1)}
    \draw[fill=black] \Point circle (0.03);
    \end{tikzpicture}
    \caption{An example of effective non-zigzag cover
    constructed in the case when $l(\lambda_o)=l(\mu_o)=3$.}
    \label{fig:effective-nonzigzag1}
\end{figure}

\begin{figure}[h]
    \centering
    \begin{tikzpicture}
    \draw[line width=0.4mm] (-4,2)--(-2.9,1.8)--(7,1.8);
    \draw[line width=0.3mm,gray] (-2.9,1.8)--(-2.5,1.5);
    \draw[line width=0.4mm] (-4,1.3)--(-2.5,1.5)--(7,1.5);
    \draw[line width=0.3mm,gray] (-2.2,1.5)--(-1.8,1.2);
    \draw[line width=0.4mm] (-4,0.5)--
    (-1.8,1.2);
    \draw[line width=0.4mm] (-1.8,1.2)--(4,0.8)--
    (1.8,0.4)--(3.8,0.2)--(1.9,0)--(4.1,-0.2)--(2,-0.4)--(7,-0.6);
    \draw[line width=0.3mm,gray] (-4,0.4)--(-1.3,0.4);
    \draw[line width=0.3mm,gray] (-4,-0.4)--(-1.3,-0.4);
    \draw[line width=0.3mm,gray] (-1.3,0)--(-4,0);
    \draw[line width=0.3mm,gray] (4,0.8)--(4.2,0.8);
    \draw[line width=0.3mm,gray] (1.8,0.4)--(1.7,0.4);
    \draw[line width=0.3mm,gray] (3.8,0.2)--(4.2,0.2);
    \draw[line width=0.3mm,gray] (1.9,0)--(1.7,0);
    \draw[line width=0.3mm,gray] (4.1,-0.2)--(4.2,-0.2);
    \draw[line width=0.3mm,gray] (2,-0.4)--(1.7,-0.4);
    \draw[line width=0.3mm,gray,dotted] (4.2,0.8)--(4.5,0.8);
    \draw[line width=0.3mm,gray,dotted] (4.2,0.2)--(4.5,0.2);
    \draw[line width=0.3mm,gray,dotted] (4.2,-0.2)--(4.5,-0.2);
    \draw[line width=0.3mm,gray,dotted] (1.4,0.4)--(1.7,0.4);
    \draw[line width=0.3mm,gray,dotted] (1.4,-0.4)--(1.7,-0.4);
    \draw[line width=0.3mm,gray,dotted] (1.4,0)--(1.7,0);
    \draw[line width=0.3mm,gray,dotted] (-1.3,0.4)--(-1,0.4);
    \draw[line width=0.3mm,gray,dotted] (-1.3,-0.4)--(-1,-0.4);
    \draw[line width=0.3mm,gray,dotted] (-1.3,0)--(-1,0);
    \draw[line width=0.3mm,gray,dotted] (-1,0.2)--(-0.8,0.2);
    \draw[line width=0.3mm,gray,dotted] (-1,-0.2)--(-0.8,-0.2);
    \draw[line width=0.3mm,gray,dotted] (-1,-0.5)--(-0.8,-0.5);
    \draw[line width=0.3mm,gray,dotted] (-1,0.5)--(-0.8,0.5);
    \draw[line width=0.3mm,gray,dotted] (-1,0.1)--(-0.8,0.1);
    \draw[line width=0.3mm,gray,dotted] (-1,-0.1)--(-0.8,-0.1);
    \draw[line width=0.3mm,gray] (-0.8,0.1)--(0.4,0)--(-0.8,-0.1);
    \draw[line width=0.3mm,gray] (0.4,0)--(1.1,0);
    \draw[line width=0.3mm,gray,dotted] (1.1,0)--(1.4,0);
    \draw[line width=0.3mm,gray] (-0.8,-0.2)--(0.2,-0.35)--(-0.8,-0.5);
    \draw[line width=0.3mm,gray] (0.2,-0.35)--(1.1,-0.35);
    \draw[line width=0.3mm,gray,dotted] (1.1,-0.35)--(1.4,-0.35);
    \draw[line width=0.3mm,gray] (-0.8,0.2)--(0.3,0.35)--(-0.8,0.5);
    \draw[line width=0.3mm,gray] (0.3,0.35)--(1.1,0.35);
    \draw[line width=0.3mm,gray,dotted] (1.4,0.35)--(1.1,0.35);
    \draw[line width=0.3mm,gray] (5.1,0.6) arc[start angle=0, end angle=360, x radius=0.15, y radius=0.1];
    \draw[line width=0.3mm,gray,dotted]  (4.5,0.6)--(4.7,0.6);
    \draw[line width=0.3mm,gray] (4.7,0.6)--(4.8,0.6);
    \draw[line width=0.3mm,gray] (5.1,0.6)--(5.2,0.6);
    \draw[line width=0.3mm,gray] (5.5,0.6) arc[start angle=0, end angle=360, x radius=0.15, y radius=0.1];
    \draw[line width=0.3mm,gray,dotted]  (5.6,0.6)--(5.8,0.6);
    \draw[line width=0.3mm,gray] (5.5,0.6)--(5.6,0.6);
    \draw[line width=0.3mm,gray,dotted] (5.8,0.8)--(6,0.8);
    \draw[line width=0.3mm,gray] (6,0.8)--(6.5,0.8)--(7,1);
    \draw[line width=0.3mm,gray] (6.5,0.8)--(7,0.6);
    \draw[line width=0.3mm,gray,dotted] (5.8,0.2)--(6,0.2);
    \draw[line width=0.3mm,gray] (6,0.2)--(6.3,0.2)--(7,0.4);
    \draw[line width=0.3mm,gray] (6.3,0.2)--(7,0);
    \draw (6.4,-0.4) node{$S_3$};
    \draw (2.6,1.3) node{$S_2$};
    \draw (2,2) node{$S_1$};
    \draw (-3.8,1.8) node{\tiny{$\lambda_1$}}
    (-3.8,1.2) node{\tiny{$\lambda_2$}} (-3.8,0.7) node{\tiny{$\lambda_3$}};
    \draw (6.8,1.95) node{\tiny{$1$}}
    (6.8,1.35) node{\tiny{$1$}} (6.8,-0.7) node{\tiny{$\mu_1$}};
    \draw (-1.8,1.35) node{\tiny{$E_2$}};
    \draw (-2.4,1.65) node{\tiny{$E_1$}};
    \draw[line width=0.3mm] (-4,-1)--(7,-1);
    \foreach \Point in {(-4,-1),(-2.5,-1), (-2.2,-1),(-1.8,-1),(-2.9,-1),(0.2,-1),(0.4,-1),(0.3,-1),(1.8,-1),(1.9,-1),(2,-1),(3.8,-1),(4,-1),(4.1,-1),(4.8,-1),(5.1,-1),(5.2,-1),(5.5,-1),(6.3,-1),(6.5,-1),(7,-1)}
    \draw[fill=black] \Point circle (0.03);
    \end{tikzpicture}
    \caption{An example of effective non-zigzag cover
    constructed in the case when $l(\lambda_o)=3, l(\mu_o)=1$.}
    \label{fig:effective-nonzigzag2}
\end{figure}

Case $\ak<0$:
Since $m$ is large enough, we can add $-2\ak$ ones to $\lambda_o$
and consider the partitions $(\lambda_o,1^{-2\ak})$ and $\mu_o$.
They have the same length $n=l(\mu_o)$.
Up to a reordering of the elements of $(\lambda_o,1^{-2\ak})$ and $\mu_o$,
we assume that $\lambda_1,\ldots,\lambda_n$ and $\mu_1,\ldots,\mu_n$ are two orderings of elements of $(\lambda_o,1^{-2\ak})$ and $\mu_o$ respectively such that $w_k:=\sum_{i=1}^k(\lambda_i-\mu_i)$ is different from zero for all $k=1,\ldots,n-1$, and $w=\sum_{i=1}^n\lambda_i-\sum_{i=1}^{n-1}\mu_i>0$.
The existence of such orderings is clear.

We construct a weighted and oriented graph $C$ by the same recursive procedure given in the case $\ak\geq0$.
The only difference is that the tails of the first type depicted in Figure $\ref{fig:zigzag}$ attached to the final string $S_n$ corresponds to parts of $((\lambda,1^{m+2\ak})_{o,o},(\mu,1^{m})_{o,o})$.
The resulting tropical cover $\varphi:C\to T\pb^1$ is properly mixed of type $(g,(\lambda,1^m),(\mu,1^m),\undl x)$ for sufficiently large $m$. See Figure $\ref{fig:effective-nonzigzag3}$ for an example in the case $\ak<0$.
\begin{figure}[h]
    \centering
    \begin{tikzpicture}
    \draw[line width=0.4mm] (-4,2)--(-1.8,1.8)--(7,1.8);
    \draw[line width=0.3mm,gray]
    (-2.2,1.5)--(-1.8,1.8);
    \draw[line width=0.4mm] (-4,1.3)--(-2.5,1.5)--(7,1.5);
    \draw[line width=0.3mm,gray] (-2.9,1.2)--(-2.5,1.5);
    \draw[line width=0.4mm] (-4,0.8)--
    (-2.9,1.2);
    \draw[line width=0.4mm] (-2.9,1.2)--(4,0.8)--
    (1.8,0.4)--(3.8,0.2)--(1.9,0)--(4.1,-0.2)--(2,-0.4)--(7,-0.6);
    \draw[line width=0.3mm,gray] (-4,0.4)--(-1.3,0.4);
    \draw[line width=0.3mm,gray] (-4,-0.4)--(-1.3,-0.4);
    \draw[line width=0.3mm,gray] (-1.3,0)--(-4,0);
    \draw[line width=0.3mm,gray] (4,0.8)--(4.2,0.8);
    \draw[line width=0.3mm,gray] (1.8,0.4)--(1.7,0.4);
    \draw[line width=0.3mm,gray] (3.8,0.2)--(4.2,0.2);
    \draw[line width=0.3mm,gray] (1.9,0)--(1.7,0);
    \draw[line width=0.3mm,gray] (4.1,-0.2)--(4.2,-0.2);
    \draw[line width=0.3mm,gray] (2,-0.4)--(1.7,-0.4);
    \draw[line width=0.3mm,gray,dotted] (4.2,0.8)--(4.5,0.8);
    \draw[line width=0.3mm,gray,dotted] (4.2,0.2)--(4.5,0.2);
    \draw[line width=0.3mm,gray,dotted] (4.2,-0.2)--(4.5,-0.2);
    \draw[line width=0.3mm,gray,dotted] (1.4,0.4)--(1.7,0.4);
    \draw[line width=0.3mm,gray,dotted] (1.4,-0.4)--(1.7,-0.4);
    \draw[line width=0.3mm,gray,dotted] (1.4,0)--(1.7,0);
    \draw[line width=0.3mm,gray,dotted] (-1.3,0.4)--(-1,0.4);
    \draw[line width=0.3mm,gray,dotted] (-1.3,-0.4)--(-1,-0.4);
    \draw[line width=0.3mm,gray,dotted] (-1.3,0)--(-1,0);
    \draw[line width=0.3mm,gray,dotted] (-1,0.2)--(-0.8,0.2);
    \draw[line width=0.3mm,gray,dotted] (-1,-0.2)--(-0.8,-0.2);
    \draw[line width=0.3mm,gray,dotted] (-1,-0.5)--(-0.8,-0.5);
    \draw[line width=0.3mm,gray,dotted] (-1,0.5)--(-0.8,0.5);
    \draw[line width=0.3mm,gray,dotted] (-1,0.1)--(-0.8,0.1);
    \draw[line width=0.3mm,gray,dotted] (-1,-0.1)--(-0.8,-0.1);
    \draw[line width=0.3mm,gray] (-0.8,0.1)--(0.4,0)--(-0.8,-0.1);
    \draw[line width=0.3mm,gray] (0.4,0)--(1.1,0);
    \draw[line width=0.3mm,gray,dotted] (1.1,0)--(1.4,0);
    \draw[line width=0.3mm,gray] (-0.8,-0.2)--(0.2,-0.35)--(-0.8,-0.5);
    \draw[line width=0.3mm,gray] (0.2,-0.35)--(1.1,-0.35);
    \draw[line width=0.3mm,gray,dotted] (1.1,-0.35)--(1.4,-0.35);
    \draw[line width=0.3mm,gray] (-0.8,0.2)--(0.3,0.35)--(-0.8,0.5);
    \draw[line width=0.3mm,gray] (0.3,0.35)--(1.1,0.35);
    \draw[line width=0.3mm,gray,dotted] (1.4,0.35)--(1.1,0.35);
    \draw[line width=0.3mm,gray] (5.1,0.6) arc[start angle=0, end angle=360, x radius=0.15, y radius=0.1];
    \draw[line width=0.3mm,gray,dotted]  (4.5,0.6)--(4.7,0.6);
    \draw[line width=0.3mm,gray] (4.7,0.6)--(4.8,0.6);
    \draw[line width=0.3mm,gray] (5.1,0.6)--(5.2,0.6);
    \draw[line width=0.3mm,gray] (5.5,0.6) arc[start angle=0, end angle=360, x radius=0.15, y radius=0.1];
    \draw[line width=0.3mm,gray,dotted]  (5.6,0.6)--(5.8,0.6);
    \draw[line width=0.3mm,gray] (5.5,0.6)--(5.6,0.6);
    \draw[line width=0.3mm,gray,dotted] (5.8,0.8)--(6,0.8);
    \draw[line width=0.3mm,gray] (6,0.8)--(6.5,0.8)--(7,1);
    \draw[line width=0.3mm,gray] (6.5,0.8)--(7,0.6);
    \draw[line width=0.3mm,gray,dotted] (5.8,0.2)--(6,0.2);
    \draw[line width=0.3mm,gray] (6,0.2)--(6.3,0.2)--(7,0.4);
    \draw[line width=0.3mm,gray] (6.3,0.2)--(7,0);
    \draw (6.4,-0.4) node{$S_3$};
    \draw (2.6,1.3) node{$S_2$};
    \draw (2,2) node{$S_1$};
    \draw (-3.8,1.8) node{\tiny{$1$}}
    (-3.8,1.2) node{\tiny{$1$}} (-3.8,0.7) node{\tiny{$\lambda_1$}};
    \draw (6.8,1.9) node{\tiny{$\mu_1$}}
    (6.8,1.3) node{\tiny{$\mu_2$}} (6.8,-0.7) node{\tiny{$\mu_3$}};
    \draw (-2.4,1.35) node{\tiny{$E_2$}};
    \draw (-1.7,1.65) node{\tiny{$E_1$}};
    \draw[line width=0.3mm] (-4,-1)--(7,-1);
    \foreach \Point in {(-4,-1),(-2.5,-1), (-2.2,-1),(-1.8,-1),(-2.9,-1),(0.2,-1),(0.4,-1),(0.3,-1),(1.8,-1),(1.9,-1),(2,-1),(3.8,-1),(4,-1),(4.1,-1),(4.8,-1),(5.1,-1),(5.2,-1),(5.5,-1),(6.3,-1),(6.5,-1),(7,-1)}
    \draw[fill=black] \Point circle (0.03);
    \end{tikzpicture}
    \caption{An example of effective non-zigzag cover
    constructed in the case when $l(\lambda_o)=1, l(\mu_o)=3$.}
    \label{fig:effective-nonzigzag3}
\end{figure}

$(2)$
Let $\varphi:C\to T\pb^1$ be a properly mixed tropical cover of type $(g,(\lambda,1^m),(\mu,1^m),\undl x)$ obtained in $(1)$ for sufficiently large even integer $m$.
From the construction of graph $C$ in $(1)$,
the tails in Figure $\ref{fig:zigzag}$ attached to the final string $S_n$
are labelled by $(2\lambda_{2e},2\lambda_{2o},(\lambda,1^{m-m_0})^2_{o,o})$
and $(2\mu_{2e},2\mu_{2o},(\mu,1^{m-m_1})^2_{o,o})$,
where
$$
m_0=\left\{
\begin{aligned}
    0,~~~ &\text{ if } l(\lambda_o)\geq l(\mu_o);\\
    l(\mu_o)-l(\lambda_o), &\text{ if } l(\lambda_o)< l(\mu_o),
\end{aligned}
\right.
\text{ and }
m_1=\left\{
\begin{aligned}
    0,~~~ &\text{ if } l(\lambda_o)\leq l(\mu_o);\\
    l(\lambda_o)-l(\mu_o), &\text{ if } l(\lambda_o)> l(\mu_o).
\end{aligned}
\right.
$$
When $m$ is even and large enough, it is easy to see that
$$
t_l\geq
\frac{m}{2}-N_1 \text{ and } t_r\geq
\frac{m}{2}-N_1,
$$
where $N_1=|l(\lambda_o)-l(\mu_o)|$.

Now we analyze the asymptotic behaviors of $b_l,b_r$.
Recall that $v_{n}$ is assumed to be the intersection point of $E_{n-1}$ and
$S_n$ in the recursive procedure in $(1)$.
Let $e_1$ be the inner edge of $S_n$ which is an outgoing
edge adjacent to the vertex $v_n$.
Suppose that $\omega(e_1)=a$ and the right end of string
$S_n$ is weighted by $b\in(\mu,1^{m})_{o}$.
Let $\lambda'=(2\lambda_{2e},2\lambda_{2o},2(\lambda,1^{m-m_0})_{o,o},a)$ and $\mu'=(2\mu_{2e},2\mu_{2o},2(\mu,1^{m-m_1})_{o,o},b)$.
We consider a set $S(\lambda',\mu')$ consisting of
all sequences of the form $(k_0,k_1,k_2,\ldots,k_{h})$,
where $k_0=a$, $k_h=b$,
$k_{i+1}=k_i+\lambda'_{s}$ or
$k_{i+1}=k_i-\mu'_{t}$ (each part of $(\lambda',\mu')$ is used exactly once here) for $i\in\{1,\ldots,h-1\}$.
Let $B(\lambda',\mu')$ denote the maximal number of sign
changes that occur in such sequences.
The integer $B(\lambda',\mu')$ is the maximal number of
bends of $S_n$ that we can create in the construction of $C$ in $(1)$,
so the weighted graph $C$ can be chosen such that
$S_n\subset C$ has $B(\lambda',\mu')$ bends.
Then we obtain that
$$
b_l\geq\left\lfloor
\frac{B(\lambda',\mu')}{2}\right\rfloor,
b_r\geq\left\lfloor
\frac{B(\lambda',\mu')}{2}\right\rfloor.
$$
Let $(a,k_1,\ldots,k_{h-1},b)$ be a sequence such that
the first $s$ entries $k_1,\ldots,k_s$ are obtained by
either adding a part of $(2\lambda_{2e},2\lambda_{2o})$ to
the previous entry or subtracting a part of $(2\mu_{2e},2\mu_{2o})$ from the previous entry,
and the last $h-s-1$ entries $k_{s+1},\ldots,k_{h-1}$
are obtained by either adding a part of $2(\lambda,1^{m-m_0})_{o,o}$ to the previous entry
or subtracting a part of $2(\mu,1^{m-m_1})_{o,o}$ from the previous entry.
Moreover, every element in $(\lambda',\mu')$ is used only once here.
From the above construction, we get that $s$ depends only on
$(\lambda_{2e}, \lambda_{2o})$ and $(\mu_{2e}, \mu_{2o})$,
and it does not depend on $m$.
Obviously, this sequence is contained in $S(\lambda',\mu')$.
By arranging the order on $+2$ and $-2$ suitably,
we assume that
there is a continuous segment $k_j,k_{j+1},\ldots,k_{j+m-N}$, where $j>s+1$ is an integer, in this sequence $(a,k_1,\ldots,k_{h-1},b)$ which has the form
$$
\pm1\to\mp1\to\cdots\to\pm1\to\mp1,
$$
for a fixed integer $N>0$.
Note that $N$ is not dependent on $m$.
Therefore, we get
$$
B(\lambda',\mu')\geq(m-N),
$$
for sufficiently large $m$.
Hence
$$
b_l\geq
\frac{m}{2}-N_2,~
b_r\geq
\frac{m}{2}-N_2,
$$
where $N_2=[\frac{N}{2}]+1$.
\end{proof}


\begin{lem}\label{lem:non-zigzag-asym}
Given a properly mixed
effective non-zigzag cover $\varphi:C\to T\pb^1$ of type $(g,\lambda,\mu,\undl x)$,
the number of properly mixed covers of that type
is bounded from below by:
$$
t_l!\cdot b_l!\cdot b_r!\cdot t_r!.
$$
\end{lem}

\begin{proof}
We consider certain symmetric group actions
on the inner vertices of $C$ which are mapped to
the last $r-n_1$ simple branch points $x_{n_1+1}<\cdots<x_r$
by the cover map $\varphi$.
Let $OS$ be the ordered set of vertices of $C$
which are mapped to the segment $x_{n_1+1}<\cdots<x_{n_1+t_l}$.
Any permutation on the vertices in $OS$ produces a new order on
these vertices,
so any element in $\sal_{t_l}$ gives a new properly mixed cover map.
Similarly, any permutation on the subgroups of vertices
which are mapped to segments in $x_{n_1+1}<\cdots<x_r$ of length $b_l$, $b_r$ and $t_r$ respectively
produces a new properly mixed cover.
Therefore, the number of properly mixed cover is bounded from below by $t_l!\cdot t_r!\cdot b_l!\cdot b_r!$.
\end{proof}

\begin{prop}\label{prop:non-zigzag-asymp}
Fix $g\in\nb$, partitions $\lambda$, $\mu$ with
$|\lambda|=|\mu|$.
Suppose that $\sum_{i}\lambda_o^i\geq\sum_{i}\mu_o^i$ and
$l(\lambda_o,\mu_o)>2$.
Then the logarithmic asymptotics for $z'_{g,\lambda,\mu}(m)$ as $m\to\infty$ for $m$ even is at least $2m\log m$,
where $z'_{g,\lambda,\mu}(m)=Z'_g((\lambda,1^{m}),(\mu,1^{m}))$.
\end{prop}

\begin{proof}
From Proposition $\ref{prop:non-zigzag1}$ $(1)$,
there is a properly mixed effective non-zigzag cover $\varphi:C\to T\pb^1$ of type $(g,(\lambda,1^m),(\mu,1^m),\undl x)$ for sufficiently large even integer $m$.
It follows from Lemma $\ref{lem:non-zigzag-asym}$ and Proposition $\ref{prop:non-zigzag1}$ $(2)$ that
$$
z'_{g,\lambda,\mu}(m)\geq (\frac{m}{2}-N_1)!^2\cdot
(\frac{m}{2}-N_2)!^2.
$$
Since $\log((\frac{m}{2}-N_1)!^2\cdot
(\frac{m}{2}-N_2)!^2)\sim 2m\log m$ as $m\to\infty$,
we obtain the required estimate of the logarithmic asymptotics for $z'_{g,\lambda,\mu}(m)$.
\end{proof}

\begin{proof}[Proof of Theorem $\ref{thm:asymptotic}$]
Let $H^\cb_g(m)=H^\cb_g((1)^m,(1)^m)$.
From \cite[Equation $5$]{dyz-2017} and the proof of \cite[Theorem $5.10$]{rau2019},
we know $\log h^\cb_{g,\lambda,\mu}(m)\leq
\log H^\cb_g(m)\sim 2m\log m$.
We assume that the partitions $\lambda$ and $\mu$ satisfy $\sum_{i}\lambda_o^i\geq\sum_{i}\mu_o^i$.
When $m$ is even, the statement follows from \cite[Theorem $5.10$]{rau2019} for $l(\lambda_o,\mu_o)\leq2$ and from Proposition $\ref{prop:non-zigzag-asymp}$ for $l(\lambda_o,\mu_o)>2$.
\end{proof}

\begin{proof}[Proof of Theorem $\ref{thm:main}$]
Because of Proposition $\ref{prop:rHurwitz-symmetry1}$, one can assume without loss of generality that $\sum_i\lambda_o^i\geq\sum_i\mu_o^i$.
The statement for even $m$ follows from Theorem $\ref{thm:asymptotic}$.
Applying the same argument to the partitions $(\lambda,1)$ and $(\mu,1)$, the statement also follows for odd $m$ and hence completes the proof.
\end{proof}

\section*{Data availability statements}
Data sharing not applicable to this article as no datasets were generated or analysed during the current study.

\section*{Acknowledgments}
The work on this text was done
during the author's visit at Institut de Math\'ematiques
de Jussieu-Paris Rive Gauche.
The author would like to thank IMJ-PRG for their
hospitality and excellent working conditions.
The author is deeply grateful to Ilia Itenberg for
valuable discussions and suggestions.
The author is also very grateful to the referees for
their valuable comments and suggestions on the manuscript
that allowed him to improve the presentation and to simplify the constructions in Section $4$.
This work was supported by China Scholarship Council,
the Natural Science Foundation of Henan Province (No. 212300410287) and NSFC (No.12101565).

\appendix
\section{Real double Hurwitz numbers via factorization}
In this Appendix, we give an equivalent
description of real double Hurwitz number via symmetric group.
\begin{defi}\label{def:real-factor}
A real factorization of type $(g,\lambda,\mu;s)$ is a tuple
$(\gamma,\sigma_1,\tau_1,\ldots,\tau_r,\sigma_2)$ of
elements of the symmetric group $\sal_d$
satisfying:
\begin{itemize}
    \item $\sigma_2\cdot\tau_r\cdot\cdots\cdot\tau_1\cdot\sigma_1=\id$;
    \item $r=l(\lambda)+l(\mu)+2g-2$;
    \item $\cl(\sigma_1)=\lambda$, $\cl(\sigma_2)=\mu$,
    $\cl(\tau_i)=(2,1,\ldots,1)$, $i=1,\ldots,r$;
    \item the subgroup generated by $\sigma_1$,
    $\sigma_2$, $\tau_1,\ldots,\tau_r$ acts transitively
    on the set $\{1,\ldots,d\}$.
    \item $\gamma$ is an involution (i.e. $\gamma^2=\id$)
    satisfying:
    $\gamma\circ\sigma_1\circ\gamma=\sigma_1^{-1}$
    and
    $$
    \gamma\circ(\tau_i\circ\cdot\cdot\cdot\circ\tau_{1}\circ\sigma_1)\circ\gamma=
    (\tau_i\circ\cdot\cdot\cdot\circ\tau_{1}\circ\sigma_1)^{-1}, \text{ for } i=1,\ldots,s, \text{ and}
    $$
    $$
    \gamma\circ(\tau_j\circ\cdot\cdot\cdot\circ\tau_{s+1})\circ\gamma=
    (\tau_j\circ\cdot\cdot\cdot\circ\tau_{s+1})^{-1},
    \text {for } j=s+1,\ldots,r.
    $$
\end{itemize}
\end{defi}
We denote by $\fl^\rb(g,\lambda,\mu;s)$ the set of
all real factorizations of type $(g,\lambda,\mu;s)$.

\begin{lem}\label{lem:realDH1}
Let $g$, $d$, $\lambda$ and $\mu$ be as above, then
$$
H^\rb_g(\lambda,\mu;s)=\frac{1}{d!}|\fl^\rb(g,\lambda,\mu;s)|.
$$
\end{lem}

\begin{proof}
The proof of this Lemma is essentially the same
as the proof of
\cite[Lemma $2.3$ and Construction $2.4$]{gpmr-2015}.
So we only give a sketch here.

We fix $r$ real points
$p_1<\ldots<p_{r-s}<0<p_{r-s+1}<\ldots<p_r$ on
$\rb P^1\setminus\{\infty\}$. Let $p_0$ be a real point
such that $p_{r-s}<p_0<0$. We choose $p_0$ as the base point.
Let $l_0,l_1,\ldots,l_r$ be $r+1$ loops depicted in
Figure $\ref{fig:real-loops1}$.
\begin{figure}[h]
    \centering
    \begin{tikzpicture}
    \draw (-3,0)--(7,0);
    \foreach\Point in {(-2,0), (0,0)}
    \draw[decoration={markings, mark=at position 0.125 with {\arrow{>}}},
        postaction={decorate}
        ] \Point circle (0.3);
    \foreach \Point in {(2,0),(3,0),(6,0)}
    \draw[decoration={markings, mark=at position 0.125 with {\arrow{>}}},
        postaction={decorate}
        ] \Point circle (0.3);
    \foreach \Point in {(-2,0),(0,0),(1,0), (2,0),(3,0),(6,0)}
    \draw[fill=black] \Point circle (0.05);
    \draw (-2,-0.5) node{$p_1$};
    \draw (0,-0.5) node{$p_{r-s}$};
    \draw (1.2,-0.3) node{$p_0$};
    \draw (2,-0.5) node{$0$};
    \draw (3,-0.5) node{$p_{r-s+1}$};
    \draw (6,-0.5) node{$p_r$};
    \draw (-1,-0.1) node{$\ldots$};
    \draw (4.5,-0.1) node{$\ldots$};
    \draw (7.5,0) node{$\rb$};
    \draw[bend left,-]  (1,0) to (2,0.3) (2,0.45) node{$l_0$};
    \draw[bend left,-]  (1,0) to (2,0.8) (2.9,0.7) node{$l_1$};
    \draw[bend left,-]  (2,0.8) to (3,0.3);
    \draw[bend left,-]  (1,0) to (2,1.3);
    \draw[bend left,-]  (2,1.3) to (6,0.3) (5.5,1) node{$l_s$};
    \draw[bend left,-]  (1,0) to (0,-0.3) (0,0.5)node{$l_{s+1}$};
    \draw[bend left,-]  (1,0) to (0,-1);
    \draw[bend left,-]  (0,-1) to (-2,-0.3)
    (-2,0.5)node{$l_r$};
    \end{tikzpicture}
    \caption{Generators of $\pi_1(\cb P^1\setminus\{0,\infty,p_1,\ldots,p_r\},p_0)$.}
    \label{fig:real-loops1}
\end{figure}
It is easy to see that $l_0,l_1,\dots,l_r$ generate the
fundamental group $\pi_1(\cb P^1\setminus\{0,\infty,p_1,\ldots,p_r\},p_0)$.
The action of complex conjugation on
$\pi_1(\cb P^1\setminus\{0,\infty,p_1,\ldots,p_r\},p_0)$
is determined by:
\begin{equation}\label{eq:conjugation}
\begin{aligned}
    &\conj(l_i\cdots l_0)=(l_i\cdots l_0)^{-1},
    ~~~~~~~~~~~~0\leq i\leq s;    \\
    &\conj(l_j\cdots l_{s+1})=(l_j\cdots l_{s+1})^{-1},
    ~~~~~~~~~~s+1\leq j\leq r.
\end{aligned}
\end{equation}
A real factorization
$(\gamma,\sigma_1,\tau_1,\ldots,\tau_r,\sigma_2)$
of type $(g,\lambda,\mu;s)$ induces a real cover as
follows: From the classical Hurwitz construction
(see \cite{hurwitz-1891} or \cite[Chapter $7$]{cm-2016}),
we know that a tuple
$(\sigma_1,\tau_1,\ldots,\tau_r,\sigma_2)$ satisfying
the first four conditions in
Definition \ref{def:real-factor} induces
a cover $\pi:C\to\cb P^1$ with ramification
profiles $\lambda$ and $\mu$ over $0$ and $\infty$,
respectively, and simple ramification over $\undl p$.
Moreover, $\pi^{-1}(p_0)$ are labelled,
{\it i.e.} $\pi^{-1}(p_0)=\{q_1,\ldots,q_d\}$,
and the monodromy actions of the loops
$l_0,\ldots,l_r$ are represented by
$\sigma_1,\tau_1,\ldots,\tau_r$ respectively.
Suppose that $p\in C$ is an unramified point.
Choose a path $\alpha$ in
$\cb P^1\setminus\{0,\infty,p_1,\ldots,p_r\}$ from
$p_0$ to $\pi(p)$. Lift $\alpha$ to a path $\tilde\alpha$
in $C$ with endpoint $p$. Let $q_k$ be the starting point
of $\tilde\alpha$. Let $\beta=\conj(\alpha)$ be the
conjugated path of $\alpha$. Then lift $\beta$ to
a path $\tilde\beta$ with starting point $q_{\gamma(k)}$.
Let $\bar p$ be the endpoint of $\tilde\beta$.
We define $\tau(p)=\bar p$. The fifth condition
in Definition \ref{def:real-factor} implies that
$\tau(p)$ is well-defined. Then one can extend $\tau$
to $C$ by standard arguments. From the construction,
we know $\pi\circ\tau=\conj\circ\pi$.
Actually, this construction gives a map
$\psi:\fl^\rb(g,\lambda,\mu;s)\to\rl$,
where $\rl$ is the set of isomorphism classes of
real Hurwitz covers of type $(g,\lambda,\mu,\undl p)$.
By a similar arguments to the proof of
\cite[Lemma $2.3$]{gpmr-2015}, we have
$\psi:\fl^\rb(g,\lambda,\mu;s)/\sal_d\to\rl$ is bijective,
and $\stab_{\sal_d}(T)=\aut(T)$,
where $T\in\fl^\rb(g,\lambda,\mu;s)$ is a factorization.
Then we get Lemma \ref{lem:realDH1}.
\end{proof}


\end{document}